%% file: paper.tex
\documentclass{article}

 \usepackage[preprint]{neurips_2026}

\usepackage[utf8]{inputenc} 
\usepackage[T1]{fontenc}    
\usepackage[pagebackref=true]{hyperref}       
\renewcommand*{\backref}[1]{}
\renewcommand*{\backrefalt}[4]{%
    \ifcase #1 \relax
    \or
        (Cited on page #2.)%
    \else
        (Cited on pages #2.)%
    \fi
}
\usepackage{url}            
\usepackage{booktabs}       
\usepackage{amsfonts}       
\usepackage{nicefrac}       
\usepackage{microtype}      
\usepackage{xcolor}         
\hypersetup{ %
  colorlinks=true,
  linkcolor=linkcolor,
  citecolor=linkcolor,
  filecolor=linkcolor,
  urlcolor=linkcolor,
}
\definecolor{linkcolor}{rgb}{0,0.08,0.45} 
\newcommand{\linkcolor}{\color{linkcolor}}

\title{
    \rennalamvrtitle: Improved Time Complexity \\
    for Parallel Stochastic Optimization \\
    via Momentum-Based Variance Reduction
    }

\author{%
  Zhirayr~Tovmasyan \\
  KAUST \\
  \texttt{zhirayr.tovmasyan@kaust.edu.sa} \\
  \And
  Artavazd~Maranjyan \\
  KAUST \\
  \texttt{artavazd.maranjyan@kaust.edu.sa} \\
  \And
  Peter~Richt\'{a}rik \\
  KAUST \\
  \texttt{peter.richtarik@kaust.edu.sa} \\
}

\input{macros.tex}

\begin{document}

\maketitle

\begin{abstract}
    %
    %
    %
    %
    Large-scale machine learning models are trained on clusters of machines that exhibit heterogeneous performance due to hardware variability, network delays, and system-level instabilities.
    In such environments, time complexity rather than iteration complexity becomes the relevant performance metric for optimization algorithms.
    Recent work by \citet{tyurin2023optimal} established the first time complexity analysis for parallel first-order stochastic optimization, proposing \rennala as a time-optimal method for smooth nonconvex optimization.
    However, \rennala is fundamentally a modification of \sgd, and variance reduction techniques are known to improve the iteration complexity of \sgd.
    In this work, we investigate whether variance reduction can also improve time complexity in heterogeneous systems.
    We show that, under a mean-squared smoothness assumption, variance reduction can improve time complexity in relevant parameter regimes.
    To this end, we propose \rennalamvr, a variance-reduced extension of \rennala based on momentum-based variance reduction, and analyze its oracle and time complexity.
    We establish lower bounds for time complexity under these assumptions.
    On a stochastic quadratic benchmark, experiments with the exact method support the theory, while neural-network experiments with a practical inexact variant show similar empirical gains over \rennala.
  \end{abstract}
  \begin{table*}
    \caption{
      Comparison of time and oracle complexities of our method with the state-of-the-art parallel first-order method \rennala, together with our derived lower bounds.
      We consider a system of $n$ workers, where worker $i$ requires $\tau_i$ seconds to compute a stochastic gradient; without loss of generality, we assume
      $\tau_1 \le \tau_2 \le \cdots \le \tau_n$.
      Both time and oracle complexities are reported for achieving an $\varepsilon$--stationary point in the nonconvex setting.
      The oracle complexity counts the total number of stochastic gradient evaluations used in the optimizer updates over the entire training process.
      The stochastic gradients computed by the workers are assumed to be unbiased with bounded variance $\sigma^2$ (\Cref{ass:unbiased_bounded_variance}), i.e.,
      $\mathbb{E}_{\xi\sim\cD}[\left\|\nabla f(x;\xi) - \nabla f(x)\right\|^2] \le \sigma^2
      \quad \text{for all } x \in \mathbb{R}^d$.
      We denote by
      $\Delta \eqdef f(x^0) - f^*$
      the initial suboptimality gap, where $f^* = \inf_x f(x)$ (\Cref{ass:lower_bound}).
      Here, $L$ is the smoothness constant of $f$, i.e., 
      $\|\nabla f(x) - \nabla f(y) \|^2 \le L^2 \| x-y \|^2$,
      and $\bar L$ denotes the mean-squared smoothness constant (\Cref{ass:ms_smoothness}), defined by
      $\mathbb{E}_{\xi\sim\cD}\|\nabla f(x;\xi) - \nabla f(y;\xi)\|^2 \le \bar L^2 \| x-y \|^2$.
      All stated complexities hide universal constant factors.
      Our method outperforms \rennala in both oracle and time complexity for sufficiently small $\varepsilon$ and when $\bar L = \mathrm{O}(L)$; see \Cref{sec:time_complexity}.
      Although our method does not always match the lower bound in terms of time complexity, it matches the lower bound in oracle complexity and, in certain regimes, can also achieve optimal time complexity.
      A detailed discussion is provided in \Cref{sec:lower_bound}.
    }
    \label{table}
    \centering 
    \begin{threeparttable}  
      \resizebox{\textwidth}{!}{
        \begin{tabular}{ccc}
          \toprule
            \bf Algorithm
            & \makecell{\bf Time Complexity}
            & \makecell{\bf Oracle \\ \bf Complexity \textsuperscript{{\linkcolor (\dag)}}} \\
          \midrule
            \makecell{\rennala \\ \citep{tyurin2023optimal}}
            & ${\orange \frac{L\Delta}{\varepsilon} }
               \min\limits_{m\in[n]}\Big(\sum\limits_{i=1}^{m}\frac{1}{\tau_i}\Big)^{-1}
              \Big({\orange \frac{\sigma^2}{\varepsilon} } + m \Big)$
            & ${\orange \frac{L\Delta\sigma^2}{\varepsilon^2} }$ \\
          \midrule
            \makecell{\rennalamvr \textbf{(new)} \\ (\Cref{thm:time_complexity})}
            & $
            \Big({\orange \frac{\bar L\Delta}{\varepsilon}+\frac{\sigma}{\sqrt{\varepsilon}} }\Big)
            { \min\limits_{m\in[n]}\left(\sum\limits_{i=1}^{m}\frac{1}{\tau_i}\right)^{-1}
            \Big({\orange \frac{\sigma}{\sqrt{\varepsilon}} } + m \Big) }
            +
            { \min\limits_{m\in[n]}\left(\sum\limits_{i=1}^{m}\frac{1}{\tau_i}\right)^{-1}
            \Big({\orange \frac{\sigma^2}{\varepsilon} } + m \Big) }
            $ 
            & $\orange \frac{\bar L\Delta\sigma}{\varepsilon^{\nicefrac{3}{2}}} + \frac{\sigma^2}{\varepsilon}$ \\
          \midrule
          \midrule
            \makecell{Lower Bound \textbf{(new)} \\ (\Cref{thm:lower_bound})}
            & $\Big({\orange \frac{\bar L \Delta}{\sigma \sqrt{\varepsilon}} + 1 }\Big) 
            { \min\limits_{m\in[n]} 
            \Big(\sum\limits_{i=1}^{m}\frac{1}{\tau_i}\Big)^{-1}
            \Big({\orange \frac{\sigma^2}{\varepsilon} } + m \Big) }$
            & $\orange \frac{\bar L\Delta\sigma}{\varepsilon^{\nicefrac{3}{2}}} + \frac{\sigma^2}{\varepsilon}$ \textsuperscript{{\linkcolor (\ddag)}} \\
          \bottomrule
        \end{tabular}%
      }
      \resizebox{\textwidth}{!}{
      \begin{minipage}{\textwidth}
        \begin{tablenotes}[para,flushleft]
          \footnotesize
          \item[{\linkcolor (\dag)}]
            \textbf{Oracle complexity} counts the total number of stochastic gradient evaluations used in the optimizer updates.
            Define $T(B) = \min_{m\in[n]}\Big(\sum_{i=1}^{m}\frac{1}{\tau_i}\Big)^{-1} (B + m)$
            By \Cref{lem:time_to_collect_B} $2T(B)$ upper bounds the time required to obtain $B$ stochastic gradients asynchronously.
            Therefore, if a method has time complexity $KT(B)$, then its oracle complexity is obtained by multiplying $K$ by the number of stochastic gradient evaluations used per iteration.
            In particular, for \rennala this gives $KB$, while for \rennalamvr it gives $B_0 + 2KB$, since each iteration uses $B$ gradient pairs, i.e. $2B$ stochastic gradient evaluations, and the initialization costs $B_0$.
          \\
          \item[{\linkcolor (\ddag)}]
            This oracle complexity matches the known lower bound on oracle complexity established by \citet{arjevani2022lower}.
        \end{tablenotes}
      \end{minipage}
      }
    \end{threeparttable}
  \end{table*}
  \section{Introduction}
  Modern machine learning models have grown to a scale that makes training on a single machine impractical.
  As a result, training is performed in distributed environments spanning multiple machines with heterogeneous computation speeds, communication delays, and system variability.
  In such settings, the efficiency of optimization algorithms is determined not only by the number of iterations required for convergence, but also by their time complexity, which is critically affected by system-level heterogeneity.
  
  Traditionally, optimization methods for training machine learning models have been analyzed through the lens of iteration complexity, which measures how many algorithmic steps are needed to reach a target accuracy.
  While this metric has been instrumental for understanding algorithmic efficiency in idealized settings, it becomes insufficient in distributed and heterogeneous environments.
  In practice, two methods with similar iteration complexity can exhibit vastly different runtimes due to idle time, synchronization overhead, and other system-level bottlenecks.
  This mismatch has motivated a shift toward time complexity as a more faithful measure of algorithmic performance in real-world distributed systems.
  
  Recent work by \citet{tyurin2023optimal} derived the first lower bounds on time complexity for parallel first-order methods with smooth nonconvex losses.
  They matched this lower bound with a method called \rennala, which is a modification of classical stochastic gradient descent (\sgd) \citep{robbins1951stochastic}.
  This line of work highlights a fundamental principle: achieving optimal performance in distributed systems often requires redesigning classical algorithms to account for system heterogeneity.
  
  Beyond \sgd, \emph{variance reduction} techniques have been shown to substantially improve optimization efficiency in the nonconvex regime under iteration complexity analysis.
  Methods such as \algname{STORM} \citep{STORM} (which uses momentum variance reduction (\mvr)) or \algname{SNVRG} \citep{zhou2020stochastic} reduce gradient noise and achieve stronger convergence guarantees than vanilla \sgd under additional assumptions.
  However, despite their strong theoretical and empirical advantages, the time complexity of variance-reduced methods in heterogeneous distributed systems remains unexplored.
  
  This raises a natural question: if variance-reduced methods outperform \sgd in terms of iteration complexity, can they also achieve better time complexity, and how should they be modified to do so?
  
  In this work, we initiate a systematic study of time complexity for distributed variance-reduced optimization.
  We propose \rennalamvr, a variance-reduced extension of \rennala, and show that it achieves better time complexity under the stronger mean-squared smoothness assumption (\Cref{ass:ms_smoothness}).
  \subsection{Contributions}
  Our contributions are summarized as follows:
  \begin{itemize}
    \item We introduce \rennalamvr (\Cref{algo}), a variance-reduced distributed optimization method for heterogeneous systems under the mean-squared smoothness assumption.
    \item We derive upper bounds on its iteration and time complexity (\Cref{sec:theory}), clarifying when variance reduction can yield time complexity gains over \rennala.
    \item We establish a new lower bound on achievable time complexity in this regime (\Cref{sec:lower_bound}).
    \item We extend the time complexity analysis to the arbitrarily varying worker speeds setting (\Cref{sec:arbitrary_time}).
    \item We provide empirical results (\Cref{sec:experiments}) validating our theoretical findings and demonstrating practical performance improvements.
  \end{itemize}
  \subsection{Related Work}
  Time complexity analysis for parallel stochastic optimization has recently emerged as a principled alternative to iteration complexity, especially in heterogeneous environments where the time per iteration depends on the workers' speeds.
  Early modern treatments of time complexity for asynchronous methods---which avoid global synchronization at the cost of using stale gradients---include works by \citet{mishchenko2022asynchronous, koloskova2022sharper, alahyane2025optimizing}.
  Asynchronous optimization itself dates back to classical work such as \citet{tsitsiklis1986distributed}, and was later popularized in machine learning by \citet{recht2011hogwild, agarwal2011distributed}; see \citet{assran2020advances} for a survey.
  
  Building on these developments, \citet{tyurin2023optimal} initiated a general time complexity framework for parallel stochastic optimization and showed that several existing asynchronous methods are not time-optimal under their model.
  This led to the \rennala method, which attains an optimal time complexity.
  Subsequent works refined and extended these ideas and produced time-optimal variants of asynchronous \sgd, including \citep{maranjyan2025ringmaster, maranjyan2026ringleader}.
  Further extensions and related developments include \citep{maranjyan2025ata, wu2026optimalasynchronous}, and a comprehensive overview appears in the work of \citet{maranjyan2025thesis}.
  Among these, \citet{maranjyan2025ringmaster} is particularly relevant to our work: their \ringmaster method is an asynchronous \sgd with the same optimal time complexity as \rennala.
  Beyond the fixed-computation-time setting of \Cref{ass:fixed_time}, several papers extend the framework of \citet{tyurin2023optimal} by considering alternative time models and settings; see, e.g., \citep{tyurin2024freya, tyurin2024shadowheart, tyurin2024optimalgraph, maranjyan2025mindflayer}.
  
  The second line of work most relevant to this paper is variance reduction.
  Classical variance-reduced methods---including \algname{SVRG} \citep{johnson2013accelerating, zeyuan2016improvedSVRG}, \algname{S2GD} \citep{S2GD}, \algname{SAGA} \citep{SAGA}, \algname{SPIDER} \citep{SPIDER}, \algname{SARAH} \citep{SARAH}, \algname{JacSketch} \citep{JacSketch}, \algname{L-SVRG} \citep{L-SVRG}, \algname{SNVRG} \citep{zhou2020stochastic}, \algname{PAGE} \citep{PAGE},  and \storm \citep{STORM}---provide improved oracle complexity guarantees in stochastic nonconvex optimization.
  Despite their strong theory, they have been less common in large-scale deep learning practice; one explanation is discussed by \citet{defazio2019ineffectiveness}.
  More recently, \citet{yuan2025mars} revisited variance reduction for training large language models and proposed a practical modification inspired by \storm \citep{STORM}, reporting speedups in LLM training.
  \section{Problem Setup}
  We consider the nonconvex optimization problem
  \begin{equation}\label{eq:problem}
    \minimize_{x \in \R^d} 
    \left\{ f(x) \coloneqq \ExpSub{\xi \sim \cD}{f(x;\xi)} \right\} ,
  \end{equation}
  where $f(x; \xi)$ is the loss function evaluated on a data sample $\xi$ drawn from distribution~$\cD$, and the model is parameterized by $x \in \mathbb{R}^d$ with $d$ denoting the parameter dimensionality.
  
  We consider a distributed learning setting with $n$ workers, where each worker~$i$ has access to the same data distribution~$\cD$.
  This setting is common in data centers with either unified memory or uniformly partitioned data.
  Following the \textit{fixed computation model} \citep{mishchenko2022asynchronous}, we formalize the heterogeneous computation times as follows:
  \begin{assumption}[Fixed Computation Model]\label{ass:fixed_time}
    Each worker~$i$ requires $\tau_i$~seconds to compute one stochastic gradient~$\nabla f(x;\xi)$.
    Without loss of generality, we assume
    \begin{equation*}
      0 < \tau_1 \le \tau_2 \le \cdots \le \tau_n .
    \end{equation*}
  \end{assumption}
We assume instantaneous communication (zero latency) between workers and the server in both directions.
This is the standard modeling assumption in prior work on time complexity for distributed stochastic optimization \citep{mishchenko2022asynchronous, koloskova2022sharper, tyurin2023optimal, maranjyan2025ringmaster}, and we adopt it here in order to make a direct comparison with these results.

We stress that this is a modeling simplification rather than a claim that communication is negligible in practice.
Explicitly modeling communication costs leads to a substantially richer theoretical problem, since one must then also specify which communication-reduction mechanisms are allowed, such as compression, quantization, sparsification, or local updates.
These questions are important, but largely orthogonal to the present paper, whose goal is to understand how variance reduction affects time complexity under heterogeneous worker speeds within the standard theoretical model.
  \subsection{Assumptions}
  We make the following standard assumptions:
  \begin{assumption}[Lower boundedness]\label{ass:lower_bound}
    There exists $f^* > -\infty$ such that $f(x) \geq f^*$ for all $x \in \R^d$.
    We define $\Delta \eqdef f(x^0) - f^*,$ where $x^0$ is the starting point of the optimization methods.
  \end{assumption}
  \begin{assumption}\label{ass:unbiased_bounded_variance}
    For every $\xi$, the function $f(x; \xi)$ is differentiable with respect to its first argument $x$.
    Moreover, the stochastic gradients are unbiased and have bounded variance $\sigma^2 \geq 0$, that is,
    \begin{gather*}
       \ExpSub{\xi \sim \cD}{\nabla f(x;\xi)} = \nabla f(x), 
          \quad \forall x \in \R^d , \\
       \ExpSub{\xi \sim \cD}{\|\nabla f(x;\xi) - \nabla f(x)\|^2} \leq \sigma^2,
          \quad \forall x \in \R^d .
  \end{gather*}
  \end{assumption}
  \begin{assumption}[Mean-squared smoothness]\label{ass:ms_smoothness}
    There exists $\bar L>0$ such that for all $x,y\in\R^d$,
    \begin{equation*}
      \ExpSub{\xi\sim\cD}{\|\nabla f(x;\xi) - \nabla f(y;\xi)\|^2}
      \le \bar L^2 \| x-y \|^2 .
    \end{equation*}
  \end{assumption}
  This assumption is stronger than classical $L$--smoothness of $f$, as shown in the following lemma.
  \begin{lemma}[Proof in \Cref{proof:l_smooth}]
    \label{lem:l_smooth}
    Mean-squared smoothness (\Cref{ass:ms_smoothness}) implies that $f$ is $\bar L$--smooth, i.e.
    \begin{equation*}
      \norm{\nabla f(x)-\nabla f(y)} \le \bar L \norm{x-y}, \qquad \forall x,y\in\R^d ~.
    \end{equation*}
  \end{lemma}
  Mean-squared smoothness also implies expected similarity, a centered gradient-difference bound closely related to the notion of Hessian variance introduced by \citet{szlendak2021permutation}.
  \begin{lemma}[Proof in \Cref{proof:expected_similarity}]
    \label{lem:expected_similarity}
    Mean-squared smoothness (\Cref{ass:ms_smoothness}) implies expected similarity with the same constant $\bar{L}$:
    \begin{align*}
      \E{ \sqnorm{ \nabla f(x;\xi) - \nabla f(y;\xi) - (\nabla f(x) - \nabla f(y)) } } \leq \bar{L}^2 \sqnorm{x-y}, \quad \forall\, x,y \in \R^d .
    \end{align*}
  \end{lemma}
  Under these assumptions, our objective is to find an $\varepsilon$--stationary point: a (possibly random) point $x$ satisfying $\mathbb{E}[\|\nabla f(x)\|^2] \leq \varepsilon$.
  \section{Background and Motivation}
  Under Assumptions~\ref{ass:lower_bound} and~\ref{ass:unbiased_bounded_variance}, combined with $L$--smoothness of $f$ (instead of \Cref{ass:ms_smoothness}), the \sgd method with iterations $x^{k+1} = x^k - \gamma \nabla f(x^k;\xi^k)$ and i.i.d.\ $\xi^k \sim \cD$ achieves optimal oracle complexity \citep{ghadimi2013stochastic,arjevani2022lower} of $\mathrm{O} (\nicefrac{L \Delta}{\varepsilon} + \nicefrac{\sigma^2 L \Delta}{\varepsilon^2})$.
  Oracle complexity measures the total number of stochastic gradient evaluations; on a single machine, this corresponds to the iteration complexity.
  
  However, achieving time-optimal performance with $n$ parallel machines requires effectively parallelizing \sgd.
  \citet{tyurin2023optimal} addressed this question by proposing the \rennala method, which is a minibatch extension of \sgd.
  Instead of computing a single gradient, \rennala collects $B$ gradients and takes a step using their average.
  The key to achieving optimal time complexity is distributing the batch of size $B$ across all available machines and computing it asynchronously (as quickly as possible).
  By setting $B=\max\{1,\nicefrac{\sigma^2}{\varepsilon}\}$, \rennala attains the lower bound on time complexity.
  Importantly, the oracle complexity---the total number of gradient evaluations---remains $\mathrm{O} (\nicefrac{\sigma^2 L \Delta}{\varepsilon^2})$, matching the lower bound up to a constant factor (see the last column in \Cref{table}).
  
  When we additionally invoke \Cref{ass:ms_smoothness}, which is stronger than $L$--smoothness of $f$ (\Cref{lem:l_smooth}), we can achieve better oracle complexity on a single machine using variance reduction methods such as \storm \citep{STORM} or \algname{SNVRG} \citep{zhou2020stochastic}.
  These methods improve the oracle complexity to $\mathrm{O} (\nicefrac{\bar L\Delta\sigma}{\varepsilon^{\nicefrac{3}{2}}} + \nicefrac{\sigma^2}{\varepsilon} )$ and are optimal \citep{arjevani2022lower}.
  
  This motivates our central question: does time complexity also improve under the stronger \Cref{ass:ms_smoothness}, and which method can we use to surpass \rennala's time complexity?
  We focus on the \storm algorithm \citep{STORM} and try to make use of its \mvr technique because it has favorable properties and is easier to parallelize across multiple machines.
  \section{\rennalamvrtitle}
  \label{sec:algo}
  We first recall the \mvr technique of \citet{STORM}, the main technique behind their \storm algorithm.
  Given a starting point $x^0\in\R^d$, an initial gradient estimator $g^0$, a momentum parameter $p\in(0,1]$, and a stepsize $\gamma>0$, The \mvr update at each iteration $k$ takes the following form:
  \begin{align}\label{eq:mvr_step}
    x^{k+1} &= x^k - \gamma g^k,\\
    g^{k+1} &= \nabla f(x^{k+1};\xi^k) + (1-p)(g^k - \nabla f(x^{k};\xi^k)),
  \end{align}
  where $\xi^k$ is sampled i.i.d. from $\cD$.
  Note that setting $p=1$ recovers standard \sgd.
  
  A distinctive feature of \mvr is that it uses two stochastic gradients per iteration, evaluated at $x^k$ and at $x^{k+1}$ using the same sample $\xi^k$.
  To extend this idea to the distributed parallel setting, we follow the same paradigm as in \rennala: instead of forming a single stochastic gradient, the server continuously aggregates stochastic gradients computed by workers and constructs minibatch estimators.
  In our case, each worker computes \emph{two} gradients---one at $x^k$ and one at $x^{k+1}$---for each sampled data point, and the server forms two minibatch averages from the first $B$ arrivals.
  The resulting method is summarized in \Cref{algo}.
  \begin{algorithm}[htb]
    \caption{\rennalamvr}\label{algo}
    \begin{algorithmic}[1]
      \STATE \textbf{Input:} initial point $x^0\in\R^d$ (stored on both the server and the workers), stepsize $\gamma>0$, minibatch size $B\in\{1,2,\ldots\}$, momentum parameter $p\in(0,1]$, initial batch size $B_0$
      \STATE Compute initial gradient estimator $g^0$ asynchronously using all workers with batch size $B_0$
      \FOR{$k = 0, \dots, K - 1$}
        \STATE Update the model: $x^{k+1} \coloneqq x^k - \gamma g^k$
        \STATE Broadcast $x^{k+1}$ to workers (workers keep $x^k$ cached), then workers compute gradients at both $x^k$ and $x^{k+1}$
        \STATE Initialize $g^{-} \coloneqq 0$, $g^{+} \coloneqq 0$ and $b \coloneqq 0$
        \WHILE{$b < B$}
          \STATE Gradients $\nabla f\big(x^k; \xi_{i^{k,b}}^{k,b}\big)$ and $\nabla f\big(x^{k+1}; \xi_{i^{k,b}}^{k,b}\big)$ arrive from worker $i^{k,b}$
          \STATE $g^{-} \leftarrow g^{-} + \nabla f\big(x^k; \xi_{i^{k,b}}^{k,b}\big)$ \\ 
                 $g^{+} \leftarrow g^{+} + \nabla f\big(x^{k+1}; \xi_{i^{k,b}}^{k,b}\big)$
          \STATE Worker $i^{k,b}$ immediately begins computing new gradients at $x^k$ and $x^{k+1}$
          \STATE $b \leftarrow b + 1$
        \ENDWHILE
        \STATE $g^{k+1} \coloneqq \frac{g^+}{B} + (1-p) (g^k - \frac{g^-}{B})$
      \ENDFOR
    \end{algorithmic}
  \end{algorithm}
  \subsection{Algorithm Description}
  The algorithm begins by obtaining an initial gradient estimator $g^0$ (for example, a minibatch average computed at $x^0$).
  At iteration $k$, the server first updates the model via
  $x^{k+1}=x^k-\gamma g^k$,
  and then ensures that the workers have access to the pair $(x^k,x^{k+1})$.
  In practice, it suffices to broadcast only the new point $x^{k+1}$ since $x^k$ was broadcast in the previous iteration and can be cached by the workers.
  
  Workers operate asynchronously and continuously.
  Each time a worker samples $\xi$, it computes the gradient pair
  $\bigl(\nabla f(x^k;\xi), \nabla f(x^{k+1};\xi)\bigr)$
  and sends this pair to the server.
  Thus, one arrival to the server consists of \emph{two} stochastic gradients corresponding to the same sample $\xi$.
  The server collects the first $B$ such arrivals, that is, $B$ gradient pairs, and forms the sums $g^-$ and $g^+$ from the gradients evaluated at $x^k$ and $x^{k+1}$, respectively.
  Equivalently, each iteration uses $B$ gradient pairs, or $2B$ stochastic gradient evaluations.
  
  Once these $B$ arrivals have been collected, the server constructs the next estimator as
  $$
    g^{k+1} = \frac{g^+}{B} + (1-p)\left(g^k - \frac{g^-}{B}\right),
  $$
  and then proceeds to the next iteration.
  \subsection{Connection to \rennalatitle}
  When $p=1$, the update simplifies to
  $g^{k+1} = \nicefrac{g^+}{B}$,
  so the method reduces to minibatch \sgd as a special case, using only gradients at the current point $x^{k+1}$.
  In this regime, the additional computation of gradients at $x^k$ (i.e., the $g^-$ term) is unnecessary.
  If we therefore modify the worker routine to compute only $\nabla f(x^{k+1};\xi)$, the resulting asynchronous minibatch \sgd implementation coincides with \rennala \citep{tyurin2023optimal}.
  \section{Theoretical Results}\label{sec:theory}
  This section presents the main theoretical guarantees for \rennalamvr (\Cref{algo}).
  We first establish its iteration complexity and the resulting \emph{oracle complexity}---the total number of stochastic gradient evaluations computed by the algorithm.
  We then derive a time complexity bound, which is the more informative metric in parallel settings with heterogeneous workers.
  \subsection{Iteration Complexity}\label{sec:iter_complexity}
  We begin with an iteration complexity result under our standard assumptions.
  \begin{theorem}[Iteration Complexity; Proof in \Cref{proof:iteration_complexity}]\label{thm:iteration_complexity}
    Under Assumptions~\ref{ass:lower_bound}--\ref{ass:ms_smoothness}, let the stepsize in \rennalamvr (\Cref{algo}) be
    $\gamma = \nicefrac{1}{4\bar{L}}$.
    Fix $\varepsilon>0$ and assume $\varepsilon < \sigma^2$ and $\varepsilon < 2\bar{L}\Delta$.
    Choose 
    $B=\left\lceil \nicefrac{6\sigma}{\sqrt{\varepsilon}}\right\rceil$,
    $p = \nicefrac{\sqrt{\varepsilon}}{\sigma}$
    and
    $B_0 = \left\lceil \nicefrac{6\sigma^2}{\varepsilon}\right\rceil$.
    Then,
    \begin{equation*}
      \frac{1}{K} \sum_{k=0}^{K-1} \E{ \sqnorm{ \nabla f(x^k) } } \ \leq \varepsilon ~,
    \end{equation*}
    for
    \begin{equation*}
      K \ge \frac{24\Delta\bar L}{\varepsilon} + \frac{\sigma}{\sqrt{\varepsilon}} ~.
    \end{equation*}
  \end{theorem}
  Our method can be viewed as a minibatch variant of the momentum variance-reduction (\mvr) mechanism used in \storm \citep{STORM}.
  The analysis is correspondingly simpler here, since we focus on the \mvr component and do not require the additional algorithmic features of \storm.
  
  A direct corollary of \Cref{thm:iteration_complexity} is an oracle complexity bound.
  The algorithm runs for $K$ iterations; each iteration computes $2B$ stochastic gradients, and the initialization costs $B_0$ stochastic gradients.
  Hence, the total oracle complexity (i.e., the number of stochastic gradient evaluations) is
  \begin{equation}\label{eq:rennalamvr_oracle_complexity}
    B_0 + 2B \cdot K
    \;=\;
    \mathrm{O} \left( \frac{\bar L\Delta\sigma}{\varepsilon^{\nicefrac{3}{2}}} + \frac{\sigma^2}{\varepsilon} \right).
  \end{equation}
  This matches the lower bound for this problem class established by \citet{arjevani2022lower}, and therefore \rennalamvr is oracle-optimal (up to constant factors) under Assumptions~\ref{ass:lower_bound}--\ref{ass:ms_smoothness}.
  
  For comparison, in the classical $L$--smooth setting (replacing mean-squared smoothness Assumption~\ref{ass:ms_smoothness}), \rennala achieves oracle complexity
  $$
    \mathrm{O} \left( \frac{L\Delta\sigma^2}{\varepsilon^2} \right),
  $$
  which is also optimal in that setting \citep{arjevani2022lower}; see \Cref{table}.
  Nevertheless, this dependence on $\varepsilon$ is worse than the oracle complexity of \rennalamvr in \eqref{eq:rennalamvr_oracle_complexity}.
  \subsection{Time Complexity}
  \label{sec:time_complexity}
  We now turn to time complexity.
  In parallel optimization with heterogeneous workers, iteration complexity alone can be misleading, since the time per iteration depends on how quickly workers return gradients.
  \begin{theorem}[Time Complexity; Proof in \Cref{proof:time_complexity}]\label{thm:time_complexity}
    Under the assumptions and parameter choices of \Cref{thm:iteration_complexity}, the time complexity of \rennalamvr (\Cref{algo}) is given by
    \begin{align*}
      T
      = \mathrm{O} \left( \left( \frac{\bar L\Delta}{\varepsilon} + \frac{\sigma}{\sqrt{\varepsilon}} \right) \min_{m\in[n]} \left(\sum_{i=1}^{m}\frac{1}{\tau_i} \right)^{-1} \left( \frac{\sigma}{\sqrt{\varepsilon}} + m \right) + \min_{m\in[n]} \left(\sum_{i=1}^{m}\frac{1}{\tau_i} \right)^{-1}\left( \frac{\sigma^2}{\varepsilon} + m \right) \right) ~.
    \end{align*}
  \end{theorem}
  Define
  $$
    T(B) \;=\; \min_{m\in[n]}\left( \sum_{i=1}^{m}\frac{1}{\tau_i} \right)^{-1} (B + m) ~,
  $$
  which upper bounds the time required to obtain $B$ stochastic gradients asynchronously (\Cref{lem:time_to_collect_B}).
  Using this notation, the bound in \Cref{thm:time_complexity} can be written more compactly as
  $$
    T_{\mvr} 
    = \mathrm{O} \left( \left( \frac{\bar L\Delta}{\varepsilon} + \frac{\sigma}{\sqrt{\varepsilon}} \right) \cdot T \left(\frac{\sigma}{\sqrt{\varepsilon}}\right) 
    + T\left( \frac{\sigma^2}{\varepsilon} \right) \right).
  $$
  \paragraph{Comparison with \rennalatitle.}
  For \rennala, the corresponding time complexity is
  $$
    T_{\sgd} 
    = \mathrm{O} \left( \frac{L\Delta}{\varepsilon} \cdot T \left( \frac{\sigma^2}{\varepsilon} \right) \right).
  $$
  The key distinction is the per-iteration waiting time.
  Each iteration of \rennalamvr waits for roughly $\nicefrac{\sigma}{\sqrt{\varepsilon}}$ gradient pairs and therefore costs on the order of
  $T ( \nicefrac{\sigma}{\sqrt{\varepsilon}} )$ time, whereas each iteration of \rennala costs
  $T ( \nicefrac{\sigma^2}{\varepsilon} )$, which is always at least as large.
  This smaller per-iteration time is the main source of potential speedups for \rennalamvr.
  
  On the other hand, the dominant iteration count for \rennalamvr (for sufficiently small $\varepsilon$) scales as
  $\nicefrac{\bar L\Delta}{\varepsilon}$, which can be larger than the $\nicefrac{L\Delta}{\varepsilon}$ scaling of \rennala since $L \le \bar L$.
  Consequently, when $\bar L = \mathrm O(L)$, the reduction in per-iteration time dominates and \rennalamvr enjoys an improved overall time complexity compared to \rennala.
  %
  %
  \paragraph{Comparison with the lower bound.}
  We next relate \Cref{thm:time_complexity} to the time lower bound proved in the following section.
  The lower bound takes the form (\Cref{thm:lower_bound})
  \begin{align*}
    T = \Omega \left(
      \left( \frac{\bar L \Delta}{\sigma \sqrt{\varepsilon}} + 1 \right) \cdot T \left( \frac{\sigma^2}{\varepsilon} \right)
      \right) ~.
  \end{align*}
  First, note that the implied oracle complexity (recall: the number of stochastic gradient evaluations) is
  $$
    \left( \frac{\bar L \Delta}{\sigma \sqrt{\varepsilon}} + 1 \right) \cdot \frac{\sigma^2}{\varepsilon}
    = \frac{\bar L\Delta\sigma}{\varepsilon^{\nicefrac{3}{2}}} + \frac{\sigma^2}{\varepsilon} ~,
  $$
  which matches the oracle complexity of \rennalamvr from \Cref{thm:iteration_complexity} (and \Cref{table}).
  In other words, \rennalamvr computes the minimal number of stochastic gradients up to constant factors.
  
  The remaining gap is a batching/synchronization issue.
  The lower bound corresponds to producing gradients in batches of size $\nicefrac{\sigma^2}{\varepsilon}$, while \rennalamvr uses smaller batches of size $\nicefrac{\sigma}{\sqrt{\varepsilon}}$.
  Smaller batches may require more frequent synchronization (i.e., more frequent termination of ongoing gradient computations), which can increase the elapsed time even when the total number of computed gradients is optimal.
  
  Despite this gap in general, the upper bound can match the lower bound in regimes where
  $$
    T \left( \frac{\sigma^2}{\varepsilon} \right) \approx \frac{\sigma}{\sqrt{\varepsilon}} \cdot T \left( \frac{\sigma}{\sqrt{\varepsilon}} \right) ~.
  $$
  For example, this occurs in the homogeneous worker case with $n=\nicefrac{\sigma}{\sqrt{\varepsilon}}$ and $\tau_i=\tau$ for all $i\in[n]$.
  
  Overall, the fact that the current time upper bound for \rennalamvr does not match the lower bound highlights a key distinction between variance-reduced methods and plain \sgd under heterogeneous compute-time models: optimal oracle complexity does not automatically imply optimal time complexity.
  Closing this gap appears to require a method that can operate with much larger effective batch sizes (on the order of $\nicefrac{\sigma^2}{\varepsilon}$) while preserving the desired convergence rate.
  Within the current proof approach and stepsize restrictions, it is unclear how to achieve this via a simple modification of \rennalamvr, which suggests that fundamentally different algorithmic ideas may be needed.
  This issue is discussed in more detail in \Cref{sec:why}.
  \section{Lower Bound on Time Complexity}
  \label{sec:lower_bound}
  In this section we formalize a framework for proving \emph{lower bounds} on the time required to find an $\varepsilon$--stationary point in stochastic nonconvex optimization with heterogeneous parallel workers.
  Following \citet{tyurin2023optimal}, we specify function class $\cF$, an oracle class $\cO$, and an algorithm class $\cA$ together with a time-based interaction protocol.
  We then define a minimax notion of time complexity and state our lower bound.
  
  Compared to the protocol in \citet{tyurin2023optimal}, our setting requires one additional feature: the algorithm may request \emph{variable batch sizes} up to a maximum of $B$ points per interaction (in particular, our \rennalamvr (\Cref{algo}) asks for two stochastic gradients).
  \subsection{Protocol}
  We consider $n$ oracles (workers) running in parallel.
  At each interaction, the algorithm returns (i) which oracle to query, (ii) the time at which it requests the reply, and (iii) a batch of up to $B$ query points.
  The protocol is given below in \Cref{protocol}.
  \begin{protocol}
    \caption{Time Multiple Oracles Protocol (variable batch size up to $B$)}
    \label{protocol}
    \begin{algorithmic}[1]
      \STATE \textbf{Input:} function $f\in\cF$, algorithm $A\in\cA$, oracles and distributions $((O_1,\ldots,O_n),(\cD_1,\ldots,\cD_n))\in\cO(f)$
      \STATE $s_i^0 = (0,0,0,0)$ for all $i\in[n]$
      \STATE $t^0 = 0$
      \FOR{$k = 0,1,2,\dots$}
        \STATE $(t^{k+1},\,{\orange i^{k+1}},\,b^k,\,X^k)=A^k(G^1,\dots,G^{k})$
        \STATE $(s^{k+1}_{{\orange i^{k+1}}},\,G^{k+1}) 
                = O_{{\orange i^{k+1}}}(t^{k+1}, b^k, X^k, s^{k}_{{\orange i^{k+1}}},\xi^{k+1})$, \hfill $\xi^{k+1}\sim \cD_{{\orange i^{k+1}}}$
        \STATE $s^{k+1}_j = s^{k}_j \quad \forall j\neq {\orange i^{k+1}}$
      \ENDFOR
    \end{algorithmic}
  \end{protocol}
  In Protocol~\ref{protocol}, $s_i^k$ denotes the internal state of oracle~$i$ at interaction~$k$ (defined precisely below).
  At interaction~$k$, the algorithm outputs an oracle index ${\orange i^{k+1}}\in[n]$, a time $t^{k+1}\ge t^k$, a batch size $b^k\in[B]$, and a tuple of query points
  $X^k=(x^{k,1},\dots,x^{k,B})\in(\R^d)^B$, with the convention $x^{k,j}=0$ for $j>b^k$.
  The oracle returns a tuple
  $G^{k+1}=(g^{k+1,1},\dots,g^{k+1,B})\in(\R^d)^B$ in which only the first $b^k$ entries may be nonzero (the remaining entries are again $0$ by convention).
  Setting $B=1$ recovers the protocol of \citet{tyurin2023optimal}.
  \subsection{Oracle Model}
  Each oracle $O_i$ has an internal state encoding whether it is idle or busy and, if busy, which batch it is currently processing.
  We represent the state as
  $$
    s_i = (s_t, s_X, s_b, s_q) \in \R_{\ge 0} \times (\R^d)^B \times \{0,\dots,B\}\times \{0,1\} ,
  $$
  where $s_q = 0$ indicates that the oracle is idle, and $s_q = 1$ indicates that it is busy computing a batch of size $s_b$ started at time $s_t$ for the stored query tuple $s_X$ (with the convention that $(s_X)_j=0$ for $j>s_b$).
  
  Formally, define
  \begin{align*}
    O_{\tau_i(\cdot),B}^{f} & :
    ~\underbrace{\R_{\ge 0}}_{\text{time}}
    \times \underbrace{\{1,\dots,B\}}_{\text{batch size}}
    \times \underbrace{(\R^d)^B}_{\text{points}} 
    \times \underbrace{(\R_{\ge 0}\times(\R^d)^B\times\{0,1,\dots,B\}\times\{0,1\})}_{\text{input state}}
    \times \underbrace{\cD_{i}}_{\text{randomness}} \\
    &\; \to~\underbrace{(\R_{\ge 0}\times(\R^d)^B\times\{0,1,\dots,B\}\times\{0,1\})}_{\text{output state}}
    \times \underbrace{(\R^d)^B}_{\text{gradients}}.
  \end{align*}
  It is defined by
  \begin{align}\label{eq:time_k-batch-oracle-variable}
    O_{\tau_i(\cdot),B}^{f}(t, b, X, (s_t,s_X,s_b,s_q), \xi)
    =
    \begin{cases}
      ( (t, X, b, 1),\ 0 ),
        & \text{if } s_q = 0, \\ 
      ( (s_t, s_X, s_b, 1),\ 0 ),
        & \text{if } s_q = 1 \text{ and } t < s_t + \tau_i(s_b), \\ 
      ( (0, 0, 0, 0),\ G ),
        & \text{if } s_q = 1 \text{ and } t \ge s_t + \tau_i(s_b),
    \end{cases}
  \end{align}
  where $0$ denotes the all-zero element of $(\R^d)^B$, and $G\in(\R^d)^B$ is given by
  $$
    G_j =
    \begin{cases}
      \nabla f((s_X)_j; \xi), & j\le s_b,\\
      0, & j>s_b.
    \end{cases}
  $$
  Thus, if the oracle is idle, it starts computing the gradients for the submitted batch (of size $b$) and becomes busy.
  While busy, it returns the all-zero tuple until the computation finishes.
  Once $t$ reaches the completion time $s_t+\tau_i(s_b)$, the oracle returns the stored batch gradients and resets to idle.
  In particular, each completed reply uses a \emph{single} random sample $\xi\sim\cD_i$ shared across the $s_b$ gradients in that reply.
  
  Relative to \citet{tyurin2023optimal}, our oracle model differs in two ways.
  First, we allow $B>1$ (variable batch sizes).
  Second, the runtime is governed by a \emph{batch-time function} $\tau_i(\cdot)$ rather than a constant: $\tau_i(s_b)$ is the time required for worker~$i$ to compute $s_b$ stochastic gradients.
  
  We impose the following natural monotonicity property.
  \begin{assumption}[Batch time]\label{ass:batch-time}
    For each oracle $i\in[n]$, the batch-time function $\tau_i(\cdot)$ is nondecreasing: for any $1\le k\le \ell$, $\tau_i(k)\le \tau_i(\ell)$.
  \end{assumption}
  \subsection{Algorithm Class}
  At each interaction, the algorithm can use all previously received replies $\{G^1,\dots,G^k\}$ to select the next oracle, the request time, the batch size, and the query points.
  We also impose the natural  constraint $t^{k+1}\ge t^k$ (the algorithm cannot query in the past).
  \begin{definition}[Algorithm class]
\label{def:alg_class}
  An algorithm $A=\{A^k\}_{k=0}^{\infty}$ is a sequence of mappings such that for each $k\ge 0$,
  $$
    A^k : \bigl((\R^d)^B\bigr)^k \to \R_{\ge 0}\times [n]\times \{1,\dots,B\}\times (\R^d)^B,
  $$
  and if $(t^{k+1},\cdot)=A^k(G^1,\dots,G^k)$ and $(t^k,\cdot)=A^{k-1}(G^1,\dots,G^{k-1})$, then $t^{k+1}\ge t^k$.
  Moreover, if $A^k(\cdot)=(t^{k+1},i^{k+1},b^k,X^k)$, then $X^k=(x^{k,1},\dots,x^{k,B})$ satisfies $x^{k,j}=0$ for all $j>b^k$.
  We denote the class of all such algorithms by $\cA$.
\end{definition}
  As in \citet{tyurin2023optimal}, we restrict attention to zero-respecting algorithms.
  \begin{definition}[Zero-respecting]\label{def:zr}
  Consider an execution of Protocol~\ref{protocol}.
  For each interaction $r\ge 0$, let
  $$
    (t^{r+1},i^{r+1},b^{r},X^{r})
  $$
  be the algorithm output, where
  $$
    X^{r}=(x^{(r,1)},\dots,x^{(r,B)})\in(\R^d)^B.
  $$
  Let
  $$
    G_{\mathrm{comp}}^{r}=(g^{(r,1)},\dots,g^{(r,B)})\in(\R^d)^B
  $$
  denote the gradient tuple eventually returned when the batch submitted at interaction $r$ is completed, with the convention $g^{(r,k)}=0$ for all $k>b^r$.
  The algorithm is \emph{zero-respecting} if for all $r\ge 0$ and all $k\in[B]$,
  $$
    \operatorname{support}\bigl(x^{(r,k)}\bigr)
    \subseteq
    \bigcup_{s<r}\ \bigcup_{k'\in[B]} \operatorname{support}\bigl(g^{(s,k')}\bigr).
  $$
  We denote the class of all zero-respecting algorithms by $\cA_{\mathrm{zr}}$.
\end{definition}
  \subsection{Function and Oracle Classes}
  We define the following.
  \begin{definition}[Function class $\cF_{\Delta,\bar L}$]
    \label{def:function_class}
    We define $\cF_{\Delta,\bar L}$ as the set of all $\bar L$--smooth functions $f:\R^d\to\R$ such that
    $f(0)-\inf_x f(x)\le \Delta$.
  \end{definition}
  \begin{definition}[Oracle class]
  \label{def:oracle_class}
    For any $f\in\cF_{\Delta,\bar L}$, the class $\cO_{\tau_1(\cdot),\dots,\tau_n(\cdot)}^{\sigma^2, \bar L, B}(f)$ returns a collection of oracles
    $O_i = O_{\tau_i(\cdot),B}^{f}$ and distributions $\cD_i$ for all $i\in[n]$, where stochastic gradients $\nabla f(\cdot; \xi)$ are unbiased and $\sigma^2$-variance-bounded
    (\Cref{ass:unbiased_bounded_variance}) and satisfy mean-squared smoothness (\Cref{ass:ms_smoothness}) with constant $\bar L$.
    The oracle $O_{\tau_i(\cdot),B}^{f}$ is given by \eqref{eq:time_k-batch-oracle-variable}.
  \end{definition}
  \subsection{Time Complexity Measure}
  Let $\mathcal{P}[\cF_{\Delta,\bar L}]$ denote the set of all probability distributions over $\cF_{\Delta,\bar L}$.
  We define the minimax time complexity as follows.
\begin{definition}\label{def:time_complexity}
  For a given batch-budget $B$ and problem parameters $(\Delta,\bar L,\sigma^2)$, we define the (minimax) time complexity as the smallest time $t$ by which a zero-respecting algorithm can ensure an $\varepsilon$--stationary point in expectation, uniformly over all admissible oracle models and hard instance distributions:
  \begin{align*}
    &m_{\mathrm{time}}(B,\Delta,\bar L,\sigma^2) \\
    &\coloneqq
    \sup_{\mathcal{O}\in\cO_{\tau_1(\cdot),\dots,\tau_n(\cdot)}^{\sigma^2,\bar L,B}}
    \sup_{P_F\in\mathcal{P}[\cF_{\Delta,\bar L}]}
    \inf_{\mathsf{A}\in\cA_{\mathrm{zr}}}
    \inf \Big\{\, t\ge 0 \ \Big|\ 
    \E{ \inf_{(r,k)\in S_t} \sqnorm{ \nabla f(x^{(r,k)}) } } \le \varepsilon \Big\},
  \end{align*}
  where, for each interaction $r\ge 0$, we denote by $c^r$ the completion time of the batch submitted at interaction $r$, and define
  \begin{align*}
    S_t
    \coloneqq
    \{\, (r,k)\in \mathbb N_0\times [B] \mid 1\le k\le b^r,\ c^r\le t \,\}.
  \end{align*}
\end{definition}
  We can now state our lower bound on the time complexity under the above protocol and classes.
\begin{theorem}[Proof in \Cref{proof:lower_bound}]\label{thm:lower_bound}
  Fix $\Delta>0$, $\bar{L}>0$, $\sigma^2>0$, $0<\varepsilon<c'\bar{L}\Delta$, an integer $B\ge1$, and $n$ workers with batch-time functions $\{\tau_i(\cdot)\}_{i=1}^n$ that satisfy \Cref{ass:batch-time}.
  Write $\tau_i\coloneqq\tau_i(1)$ and assume $0<\tau_1\le\cdots\le\tau_n$.
  Let the chain constants $\Delta_0,\ell_1,\gamma_\infty$ be as in \Cref{lem:FT-props} and the estimator constants $\varsigma,\bar\ell_1$ as in \Cref{lem:gbar}.
  Define
  $$
    p \eqdef \min\!\left\{\frac{2\varepsilon \varsigma^2 }{\sigma^2},\,1\right\},
    \qquad
    L \eqdef \frac{\ell_1}{\bar\ell_1}\,\bar L \sqrt{p} ~,
    \qquad
    \lambda \eqdef \frac{\ell_1}{L}\,\sqrt{2\varepsilon}~,
    \qquad
    T \eqdef \Big\lfloor \frac{L\Delta}{2\,\Delta_0\,\ell_1\,\varepsilon} \Big\rfloor.
  $$
  Then there exist $f\in\cF_{\Delta,\bar L}$ and an oracle class $\mathcal{O}\in \cO_{\tau_1(\cdot),\dots,\tau_n(\cdot)}^{\sigma^2,\bar L, B}(f)$ such that, under Protocol~\ref{protocol},
  \begin{align*}
    m_{\mathrm{time}}(B, \Delta, \bar L, \sigma^2)\ \ge\ 
    c \cdot \left( \frac{\bar L \Delta \min\left\{\nicefrac{\sqrt{\varepsilon}}{\sigma},\,1\right\} }{\varepsilon} + 1 \right)
    \min_{m\in[n]}
    \left(\sum_{i=1}^{m}\frac{1}{\tau_i}\right)^{-1}
    \left( \frac{\sigma^2}{\varepsilon} + m \right) .
  \end{align*}
\end{theorem}
\section{Experiments}
  \label{sec:experiments}
  \begin{figure*}[t]
      \centering
      \begin{subfigure}[t]{0.32\textwidth}
          \centering
          \includegraphics[width=\textwidth]{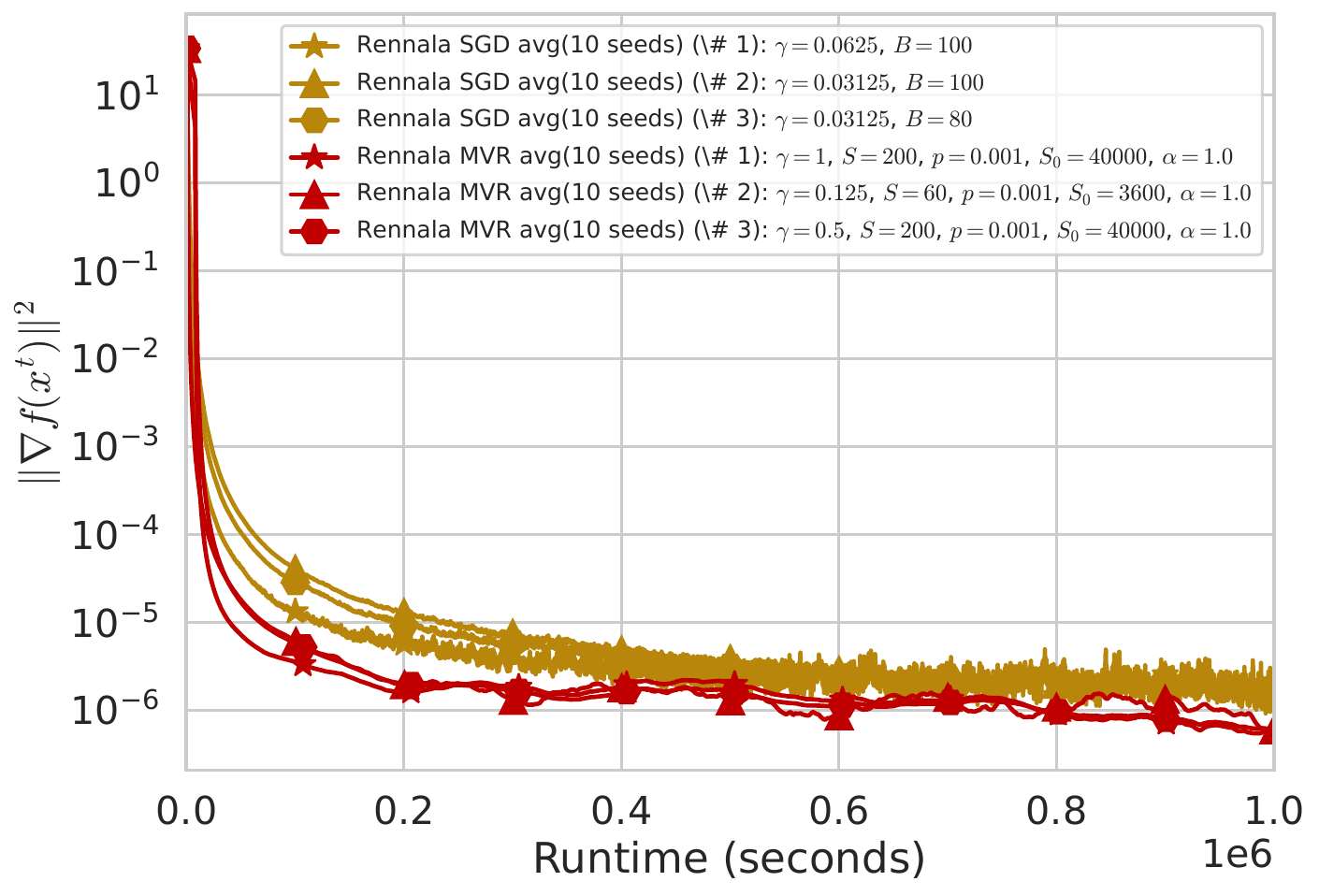}
          \caption{$\tau_i=\sqrt{i}$}
          \label{fig:quadratic_100_sqrt}
      \end{subfigure}
      \hfill
      \begin{subfigure}[t]{0.32\textwidth}
          \centering
          \includegraphics[width=\textwidth]{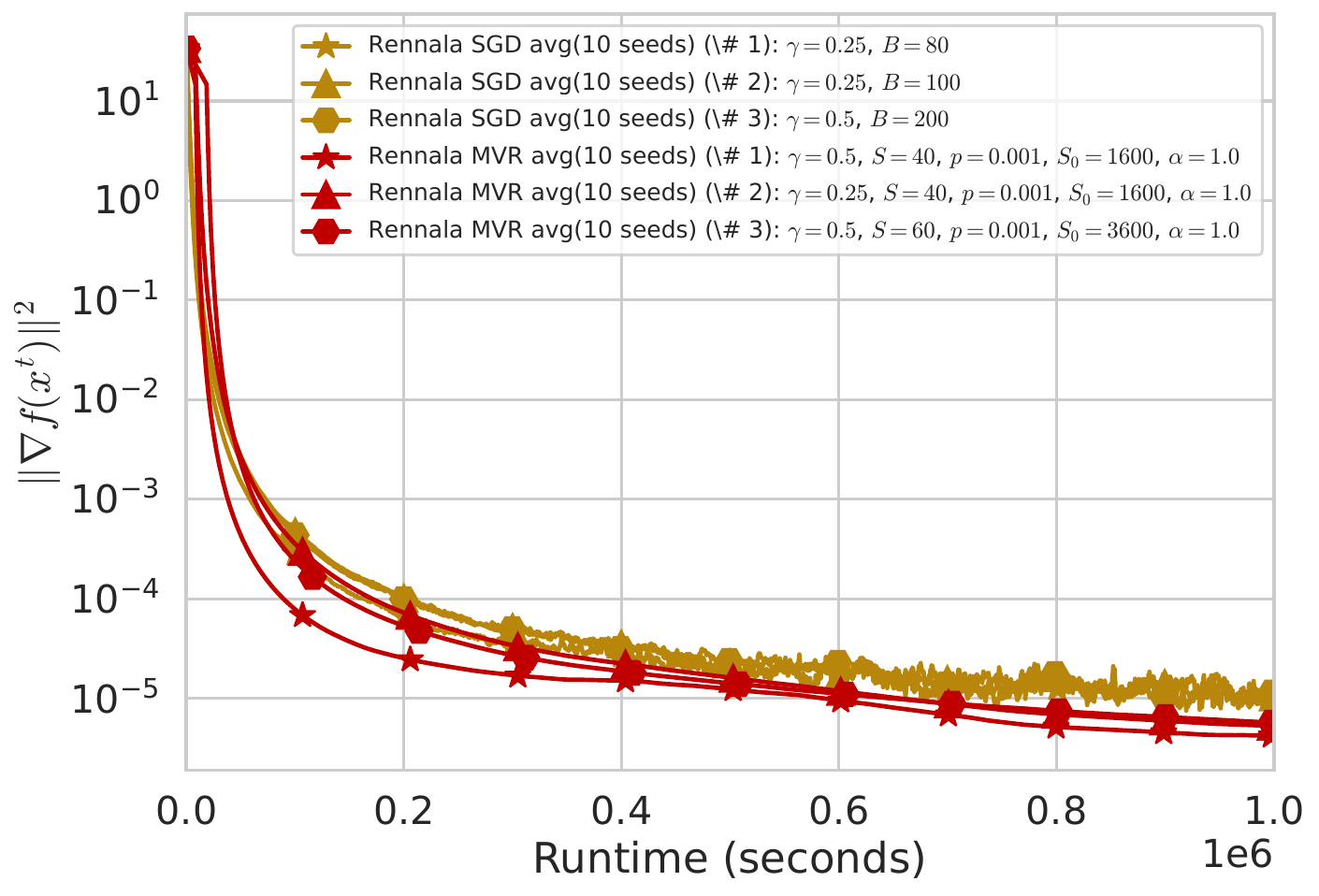}
          \caption{Mixture delays}
          \label{fig:quadratic_100_mixture}
      \end{subfigure}
      \hfill
      \begin{subfigure}[t]{0.32\textwidth}
          \centering
          \includegraphics[width=\textwidth]{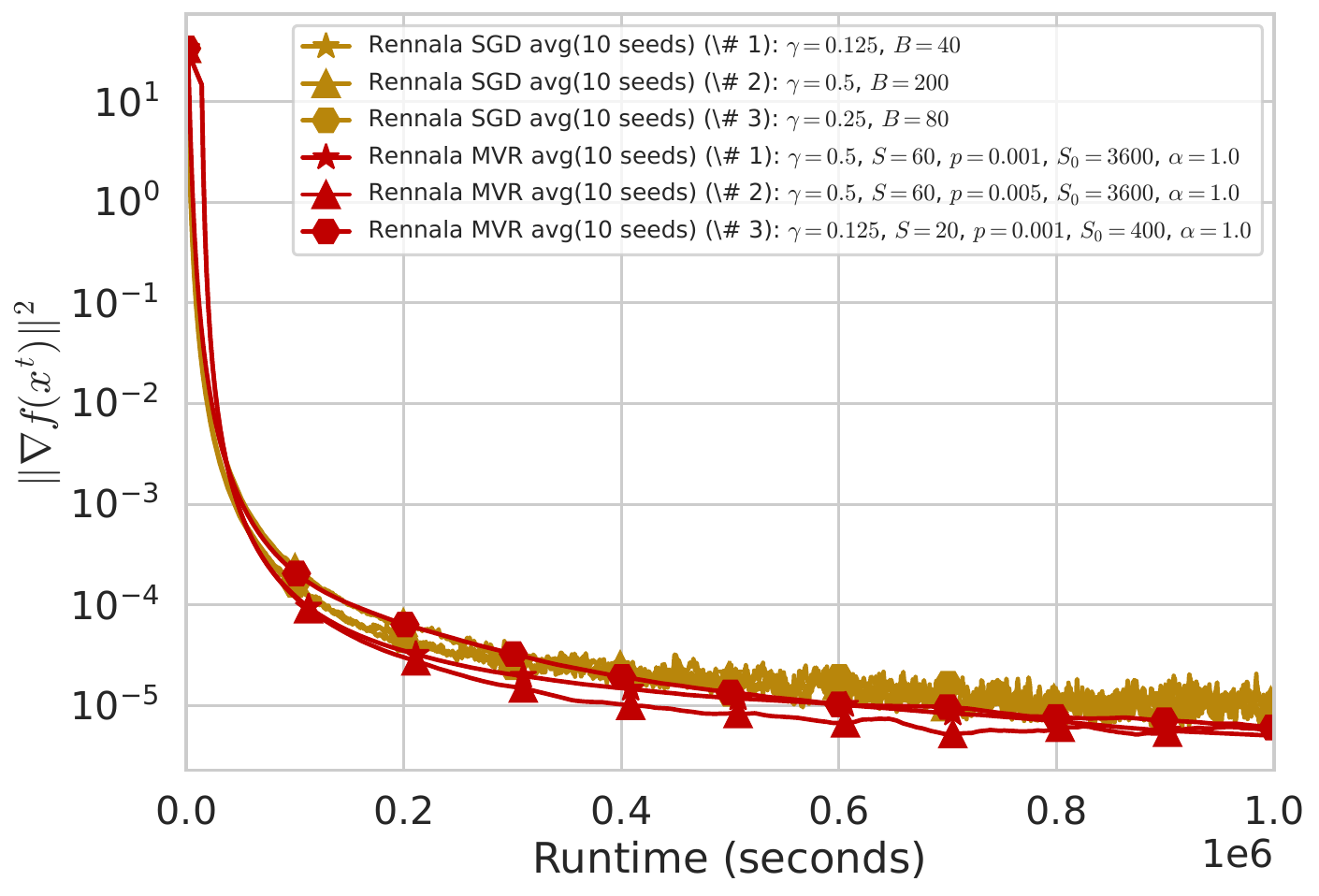}
          \caption{Uniform delays}
          \label{fig:quadratic_100_uniform}
      \end{subfigure}
      \caption{Comparison of \rennalamvr and \rennala on the stochastic quadratic benchmark with $10$ workers under three delay models.}
  \end{figure*}
  
  \begin{figure*}[t]
      \centering
      \begin{subfigure}[t]{0.32\textwidth}
          \centering
          \includegraphics[width=\textwidth]{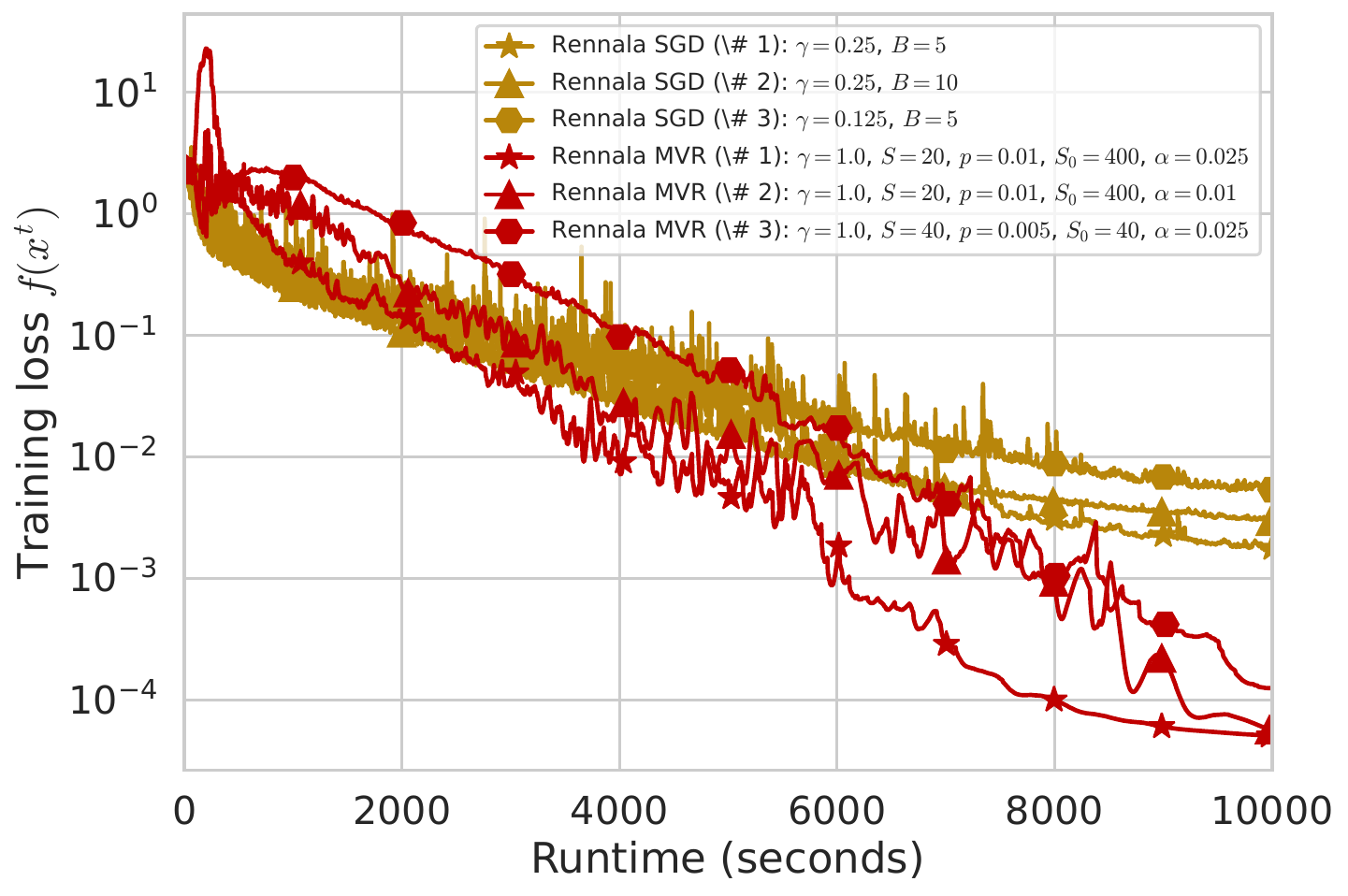}
          \caption{$\tau_i=\sqrt{i}$}
          \label{fig:nn_inexact_sqrt}
      \end{subfigure}
      \hfill
      \begin{subfigure}[t]{0.32\textwidth}
          \centering
          \includegraphics[width=\textwidth]{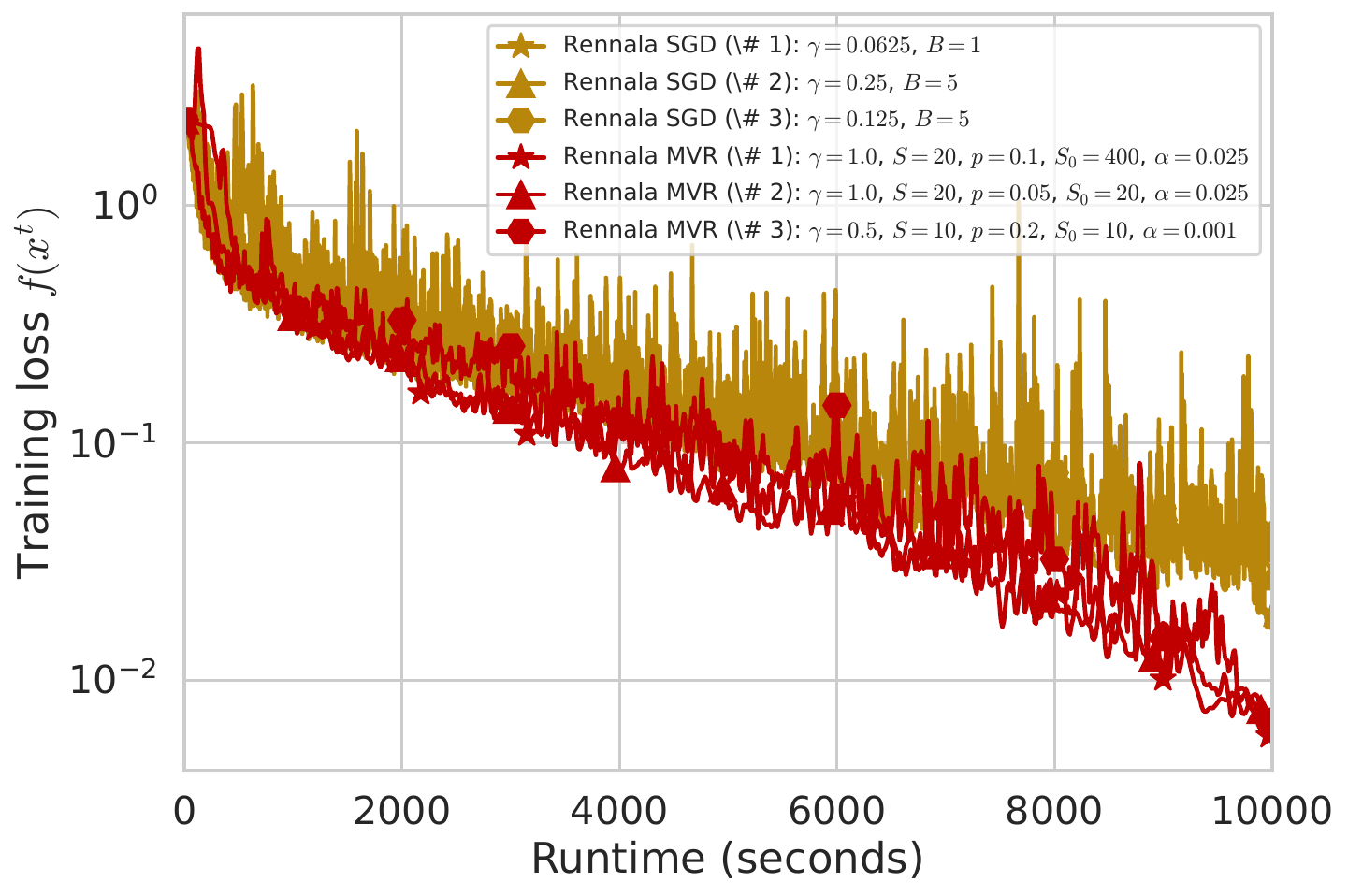}
          \caption{Mixture delays}
          \label{fig:nn_inexact_mixture}
      \end{subfigure}
      \hfill
      \begin{subfigure}[t]{0.32\textwidth}
          \centering
          \includegraphics[width=\textwidth]{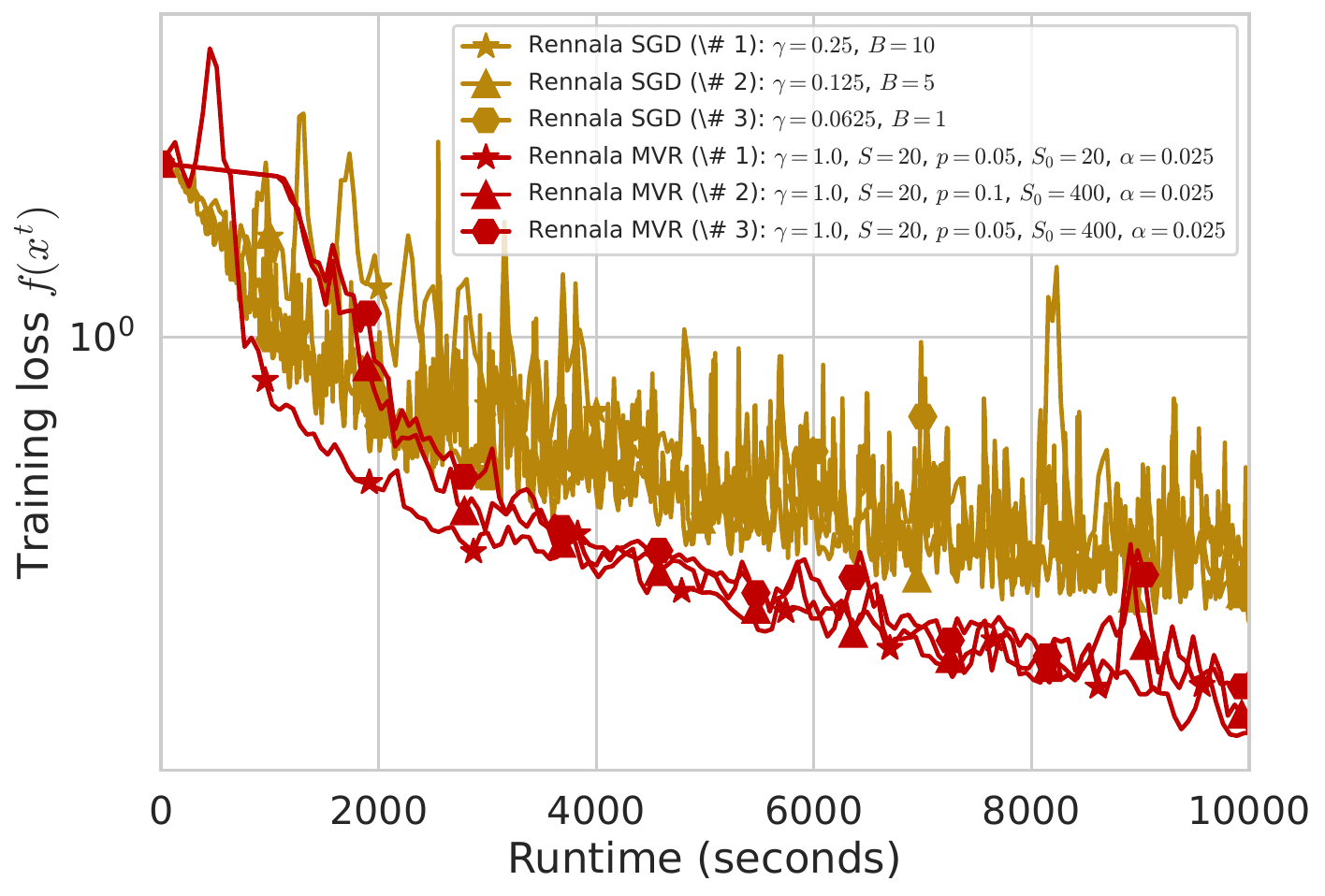}
          \caption{Uniform delays}
          \label{fig:nn_inexact_uniform}
      \end{subfigure}
      \caption{Comparison of the inexact variant of \rennalamvr and \rennala on asynchronous neural-network training over a MNIST subset with $10$ workers under three delay models.
      This figure reports experiments with the practical $\alpha$-parameterized inexact variant from \Cref{sec:nn_experiments}, not \Cref{algo} exactly.}
      \label{fig:nn_inexact_all}
  \end{figure*}
  
We evaluate our method against \rennala \citep{tyurin2023optimal} on two benchmarks: 
(i) a stochastic tridiagonal quadratic problem, where we test the exact method analyzed in the paper, and (ii) asynchronous neural-network training, where we use an inexact practical variant inspired by \algname{MARS} \citep{yuan2025mars}.
The first benchmark is theory-aligned, while the second is intended as an empirical test outside the scope of our current analysis.
  \subsection{Quadratic Benchmark}\label{sec:quadratic_experiments}
  We consider the quadratic objective
  \begin{equation*}
    f(x) = \frac{1}{2} x^\top A x - b^\top x ,
  \end{equation*}
  with dimension $d=100$, where $A\in\R^{d\times d}$ and $b\in\R^d$ are given by
  \begin{equation*}
    A = \frac{1}{4}\!\left(\!\begin{array}{cccc}
      2 & -1 & & 0\\
      -1 & \ddots & \ddots &\\
       & \ddots & \ddots & -1 \\
       0 & & -1 & 2
    \end{array}\!\right),
    \qquad
    b = \frac{1}{4}\!\left(\!\begin{array}{c}
      -1\\
      0\\
      \vdots\\
      0
    \end{array}\!\right).
  \end{equation*}
  We initialize at $x^0=(\sqrt{d},0,\dots,0)$.
  Each worker queries the same unbiased stochastic gradient oracle
  \begin{equation*}
    \nabla f(x;\zeta)
    =
    \nabla f(x) + \zeta,
  \end{equation*}
  where $\zeta \sim \mathcal{N}(0,\sigma_{\mathrm{add}}^2 \mathbb{I}_d)$ with $\sigma_{\mathrm{add}}=0.1$.
  That is, each stochastic gradient is formed by adding isotropic Gaussian noise to the exact gradient.
  
  \paragraph{Distributed setup and tuning.}
  For the quadratic experiments shown in \Cref{fig:quadratic_100_sqrt,fig:quadratic_100_mixture,fig:quadratic_100_uniform}, we model heterogeneity through fixed per-worker delays.
  We use $n=10$ workers and a simulated time budget of $10^6$.
  For each hyperparameter setting, we average over $10$ independent runs with different per-node random seeds.
  We consider three delay models: the canonical square-root profile $\tau_i=\sqrt{i}$, randomly permuted across workers; delays sampled independently from the uniform distribution on $[1,10n]$; and a three-peak Gaussian mixture, obtained by sampling three peak locations in $[1,10n]$, assigning each worker to one of the peaks, drawing its delay from the corresponding Gaussian component, and clipping the result to $[1,10n]$.
  For both \rennala and \rennalamvr, we tune the stepsize over
  \[
    \gamma \in \{2^j : j=-15,\dots,2\},
  \]
  and the minibatch size over
  \[
    B \in \{1,5,10,20,40,60,80,100,200\}.
  \]
  For \rennalamvr, we additionally tune the momentum parameter over
  \[
    p \in \{0.001,0.005,0.01,0.05,0.1,0.2,0.5,0.9\},
  \]
  and the initialization batch size over
  \[
    B_0 \in \{B,B^2\}.
  \]
  We report the stationarity measure $\|\nabla f(x)\|^2$ as a function of time.
  In all plots, we show the three best-performing hyperparameter configurations for each method, where performance is evaluated by the median value of the reported metric over the final $1\%$ of the time horizon.
  \paragraph{Results.}
The quadratic experiments in
\Cref{fig:quadratic_100_sqrt,fig:quadratic_100_mixture,fig:quadratic_100_uniform}
evaluate the exact method analyzed in this paper.
Across all three delay models, the best-tuned configurations of \rennalamvr reach lower values of \(\|\nabla f(x)\|^2\) than \rennala in the \(10\)-worker regime.
This provides empirical support for the time complexity advantage predicted by our theory.
  
\subsection{Neural Network Experiments}\label{sec:nn_experiments}
We also evaluate the practical behavior of \rennalamvr on asynchronous neural-network training.
We train a two-layer ReLU network with hidden dimension $200$ on a $3000$-example MNIST subset using $10$ heterogeneous workers, local batch size $4$, and a simulated time budget of $10^4$.
We again consider the square-root, uniform, and three-peak mixture delay models.

Unlike the quadratic benchmark in \Cref{sec:quadratic_experiments}, this experiment does \emph{not} evaluate \Cref{algo} exactly.
Instead, in the neural-network setting we study an inexact variant of \rennalamvr inspired by \algname{MARS} \citep{yuan2025mars}.
The motivation is practical: while the present network is small, for larger models the exact variant introduces additional runtime overhead, whereas the inexact variant reuses previous stochastic-gradient information and is therefore computationally more attractive.
This choice is also consistent with the empirical findings of \algname{MARS}, where the inexact variant was observed to perform comparably to the exact one in practice.
Accordingly, the neural-network experiments should be interpreted as empirical evidence about the practical behavior of variance reduction in this setting, rather than as a direct validation of the theory developed for the exact method.

We additionally introduce a scaling parameter $\alpha \in (0,1]$ and consider the update
\begin{align*}
  x^{k+1} &= x^k - \gamma g^k,\\
  g^{k+1} &= (1-p)g^k + p\nabla f(x^{k+1};\xi^k)
    + {\orange \alpha} (1-p) \left( \nabla f(x^{k+1};\xi^{k}) - \nabla f(x^{k};{\orange \xi^{k-1} }) \right).
\end{align*}
The role of $\alpha$ is to scale the correction term.
Recall, the update of the gradient estimator in our algorithm is given by \eqref{eq:mvr_step}
\begin{align*}
    g^{k+1} &= (1-p)g^k + p\nabla f(x^{k+1};\xi^k)
      + (1-p) \left(\nabla f(x^{k+1};\xi^{k}) - \nabla f(x^{k};\xi^{k})\right) ,
\end{align*}
where $\alpha = 1$ and the last stochastic gradient is computed with the same current data point $\xi^k$.
This modification is motivated by the fact that pure momentum-based variance reduction often behaves poorly in neural-network training, whereas introducing a small correction coefficient can substantially improve stability and empirical performance, as also observed in MARS and related follow-up work.

We tune the stepsize $\gamma$, the batch size $B$, the momentum parameter $p$, and the interpolation parameter $\alpha$ over the grids
\begin{align*}
  \gamma &\in \{2^j : j=-10,\dots,0\},\\
  B &\in \{1,5,10,20,40,60\},\\
  p &\in \{0.001,0.005,0.01,0.05,0.1,0.2\},\\
  \alpha &\in \{0.001,0.005,0.01,0.025\},
\end{align*}
and set the initialization batch size to $B_0 \in \{B,B^2\}$.
Here, $B$ denotes the batch size used to construct the stochastic gradient estimator, and $B_0$ denotes the batch size used for its initialization.

For these neural-network runs, we report the training loss as a function of time.
For visualization, the main figures show the best-performing hyperparameter configurations from the search grid, while Appendix~\ref{add_experiments} provides additional sensitivity plots.
Across the tested parameter ranges, performance is relatively insensitive to the auxiliary parameters $p$ and $\alpha$, which suggests that, in practice, the dominant tuning burden remains the stepsize and batch size, as in \rennala.
\paragraph{Results.}
The neural-network experiments in
\Cref{fig:nn_inexact_sqrt,fig:nn_inexact_mixture,fig:nn_inexact_uniform}
evaluate the $\alpha$-parameterized inexact variant of \rennalamvr introduced in this subsection rather than Algorithm~\ref{algo} exactly.
Although this setting lies outside the scope of our present analysis, it exhibits the same qualitative trend across all three delay models: the inexact variant achieves lower training loss than \rennala under the square-root, mixture, and uniform delay models.
These results suggest that the practical benefits of variance reduction may extend beyond the exact theory-covered setting.
\section{Conclusion}
This work initiates time complexity analysis for variance-reduced methods in heterogeneous parallel stochastic optimization.
Focusing on the \mvr mechanism in \storm \citep{STORM}, and working under the stronger mean-squared smoothness assumption, we proposed \rennalamvr and established its iteration, oracle, and time complexity guarantees.
In particular, we showed that variance reduction can improve not only oracle complexity, but also time complexity in heterogeneous parallel environments, yielding improved guarantees over \rennala \citep{tyurin2023optimal}, the previously time-optimal method in the standard smoothness regime, in relevant parameter regimes.

More broadly, our results highlight a key conceptual message: in parallel stochastic optimization, oracle complexity is not sufficient to understand performance.
Indeed, methods based on momentum-based variance reduction are known to be oracle-optimal in the single-worker setting under stronger smoothness assumptions \citep{STORM,arjevani2022lower}; however, as our results and \Cref{table} show, this does not automatically translate into time optimality in the parallel heterogeneous-worker setting.
Time complexity behavior also depends on how optimization progress interacts with the worker-time model and the batch-size/time tradeoff.
In this sense, a central novelty of the paper lies in the time complexity theory needed to understand this mechanism in the heterogeneous parallel regime.

To support this perspective, we also established a new lower bound on the achievable time complexity under the same model.
Together with our upper and lower bounds, this provides, to the best of our knowledge, the first theoretical understanding of the time complexity behavior of momentum-based variance reduction in heterogeneous distributed systems.
At the same time, a gap between the current upper and lower bounds remains.
Closing this gap appears to be nontrivial, and likely requires fundamentally different ideas rather than a routine extension of existing variance-reduction analyses.
\begin{ack}
    The research reported in this publication was supported by funding from King Abdullah University of Science and Technology (KAUST): i) KAUST Baseline Research Scheme, ii) CRG Grant ORFS-CRG12-2024-6460, and iii) Center of Excellence for Generative AI, under award number 5940.
\end{ack}
{
\small
\bibliographystyle{plainnat}
\bibliography{bib}
}
\newpage
\appendix
\tableofcontents
\section{Time Complexity Under Arbitrarily Varying Worker Speeds}\label{sec:arbitrary_time}
The fixed-computation model in \Cref{ass:fixed_time} assumes that every worker processes gradients at a constant rate.
This assumption is often too restrictive in practice: worker speeds may change over time due to interruptions, network interference, or temporary hardware slowdowns.
To capture such effects, we now replace the constants $\{\tau_i\}_{i=1}^n$ with time-varying computation rates.

\subsection{Universal Computation Model}
\label{sec:universal_computation_model}
We adopt the universal computation model of \citet{tyurin2025tight}.
\begin{assumption}[Universal Computation Model]
  \label{assump:universal_time}
  For each worker $i \in [n]$, there exists a nonnegative function
  \begin{equation*}
    p_i : \R_{+} \to \R_{+}
  \end{equation*}
  that is continuous almost everywhere.
  For any $0 \le T_1 \le T_2$, the number of stochastic gradients completed by worker $i$ during the interval $[T_1,T_2]$ is
  \begin{equation*}
    N_i(T_1,T_2) \coloneqq \left\lfloor \int_{T_1}^{T_2} p_i(s)\, ds \right\rfloor .
  \end{equation*}
\end{assumption}
Moreover, \Cref{assump:universal_time} contains \Cref{ass:fixed_time} as a special case: if $p_i(s)=1/\tau_i$ for all $s \ge 0$, then
\begin{equation*}
  N_i(T_1,T_2) = \left\lfloor \frac{T_2-T_1}{\tau_i} \right\rfloor .
\end{equation*}
\subsection{Time Complexity in the Universal Model}
Unlike the fixed-time model, the universal model does not generally lead to a closed-form expression for the total time.
The natural guarantee is therefore implicit and expressed through the times at which enough gradients have been completed to finish the initialization and each subsequent iteration.
For the universal-model upper bound, we assume that at each iteration boundary all workers are effectively reset to the idle state, so unfinished computations from the previous iteration do not contribute to the next one.
\begin{theorem}[Time Complexity Under the Universal Computation Model]
  Assume the setup and parameter choices of \Cref{thm:iteration_complexity}, and define
  \begin{equation*}
    K_\varepsilon \coloneqq \left\lceil \frac{24\Delta\bar L}{\varepsilon}
    + \frac{\sigma}{\sqrt{\varepsilon}} \right\rceil .
  \end{equation*}
  Let
  \begin{align*}
    B_0 = \left\lceil \frac{6\sigma^2}{\varepsilon}\right\rceil,
    \quad
    B = \left\lceil \frac{6\sigma}{\sqrt{\varepsilon}}\right\rceil,
  \end{align*}
  and define the completion times $\{T_{\mathrm{MVR}}^k\}_{k\ge 0}$ by
  \begin{align*}
    T_{\mathrm{MVR}}^0
    &\coloneqq
    \min\left\{ t \ge 0 : \sum_{i=1}^{n} \left\lfloor \int_0^t p_i(s)\, ds \right\rfloor \ge B_0 \right\}, \\
    T_{\mathrm{MVR}}^k
    &\coloneqq
    \min\left\{ t \ge T_{\mathrm{MVR}}^{k-1} :
    \sum_{i=1}^{n}\left\lfloor \frac{1}{2}\int_{T_{\mathrm{MVR}}^{k-1}}^{t} p_i(s)\, ds \right\rfloor \ge B \right\},
    \qquad k \ge 1 .
  \end{align*}
  Then \rennalamvr (\Cref{algo}) returns an $\varepsilon$--stationary point within at most $T_{\mathrm{MVR}}^{K_\varepsilon}$ seconds.
\end{theorem}
\begin{proof}
  By \Cref{thm:iteration_complexity}, under the stated parameter choice, \rennalamvr reaches an $\varepsilon$--stationary point after at most $K_\varepsilon$ iterations.
  It therefore remains to upper bound the time needed to complete the initialization and these $K_\varepsilon$ iterations.

  By definition of $T_{\mathrm{MVR}}^0$, by time $T_{\mathrm{MVR}}^0$ the workers have collectively produced at least $B_0$ stochastic gradients at $x^0$.
  Hence the initialization step used to construct $g^0$ is completed no later than $T_{\mathrm{MVR}}^0$.

  Now fix $k \ge 1$ and suppose that iteration $k$ starts at time $T_{\mathrm{MVR}}^{k-1}$.
  Over the interval $[T_{\mathrm{MVR}}^{k-1}, t]$, worker $i$ can complete
  \begin{equation*}
    \left\lfloor \int_{T_{\mathrm{MVR}}^{k-1}}^{t} p_i(s)\, ds \right\rfloor
  \end{equation*}
  stochastic gradients.
  One arrival in \rennalamvr at iteration $k$ is a complete pair
  \begin{equation*}
    \left(\nabla f(x^{k-1};\xi), \nabla f(x^{k};\xi)\right),
  \end{equation*}
  and therefore requires two stochastic-gradient computations on the same worker.
  Consequently, by time $t$, worker $i$ can contribute at least
  \begin{equation*}
    \left\lfloor \frac{1}{2}\int_{T_{\mathrm{MVR}}^{k-1}}^{t} p_i(s)\, ds \right\rfloor
  \end{equation*}
  complete gradient pairs.
  Summing over workers, the server has received at least
  \begin{equation*}
    \sum_{i=1}^{n}\left\lfloor \frac{1}{2}\int_{T_{\mathrm{MVR}}^{k-1}}^{t} p_i(s)\, ds \right\rfloor
  \end{equation*}
  pairs by time $t$.
  By the definition of $T_{\mathrm{MVR}}^k$, this quantity is at least $B$ when $t=T_{\mathrm{MVR}}^k$.
  Hence iteration $k$ finishes no later than $T_{\mathrm{MVR}}^k$.

  Applying the above argument inductively for $k=1,\dots,K_\varepsilon$, we conclude that all $K_\varepsilon$ iterations finish by time $T_{\mathrm{MVR}}^{K_\varepsilon}$.
  Combining this with \Cref{thm:iteration_complexity} proves the claim.
\end{proof}
\paragraph{Comparison with \rennalatitle.}
The corresponding universal-model bound for \rennala follows from \citet[Theorem~5.3]{tyurin2025tight}, specialized to our notation.
In the regime $\varepsilon < \sigma^2$ considered in \Cref{thm:iteration_complexity}, \rennala reaches an $\varepsilon$--stationary point within $T_{\mathrm{SGD}}^{K_{\mathrm{SGD}}}$ seconds, where
\begin{equation*}
  K_{\mathrm{SGD}} \coloneqq \left\lceil \frac{24L\Delta}{\varepsilon} \right\rceil
\end{equation*}
and
\begin{align*}
  T_{\mathrm{SGD}}^0
  &\coloneqq 0, \\
  T_{\mathrm{SGD}}^k
  &\coloneqq
  \min\left\{ t \ge T_{\mathrm{SGD}}^{k-1} :
  \sum_{i=1}^{n}\left\lfloor \int_{T_{\mathrm{SGD}}^{k-1}}^{t} p_i(s)\, ds \right\rfloor
  \ge \left\lceil \frac{\sigma^2}{\varepsilon} \right\rceil \right\},
  \qquad k \ge 1 .
\end{align*}
Thus, the universal model preserves the same qualitative trade-off as the fixed-time model.
Each iteration of \rennalamvr waits for $B=\Theta(\sigma/\sqrt{\varepsilon})$ complete pairs, whereas each iteration of \rennala waits for $\Theta(\sigma^2/\varepsilon)$ stochastic gradients.
For sufficiently small $\varepsilon$, the number of stochastic gradients that must be completed before each \rennalamvr update is therefore asymptotically smaller.
The trade-off is that \rennalamvr requires $K_\varepsilon=\Theta(\bar L\Delta/\varepsilon+\sigma/\sqrt{\varepsilon})$ update rounds, whereas \rennala requires $K_{\mathrm{SGD}}=\Theta(L\Delta/\varepsilon)$ rounds.
Since the universal-model bounds are implicit, no unconditional ordering between the two total times follows without additional structure on the rate functions $\{p_i\}_{i=1}^n$.
Nevertheless, whenever $\bar L=\mathrm{O}(L)$ and the reduction in per-round waiting time dominates, the same mechanism that improves the fixed-time bound can also favor \rennalamvr under time-varying worker speeds.

\section{Proof of Lemmas}
This appendix provides complete proofs of all lemmas referenced in the main text, including both those stated explicitly and auxiliary results needed for the theorem proofs.
\subsection{Proof of \Cref{lem:l_smooth}}\label{proof:l_smooth}
Let us first restate the lemma.
\begin{restate-lemma}{\ref{lem:l_smooth}}
    Mean-squared smoothness (\cref{ass:ms_smoothness}) implies smoothness of $f$ with constant $\bar{L}$, i.e.
    \begin{align*}
        \norm{\nabla f(x) - \nabla f(y)} 
        \leq \bar{L} \norm{x - y}, 
        \quad \forall\, x,y \in \R^d ~.
    \end{align*}
\end{restate-lemma}
\begin{proof}
  Using Jensen's inequality, we have
    \begin{align*}
        \sqnorm{ \nabla f(x) - \nabla f(y) }
        = \sqnorm{ \E{\nabla f(x;\xi) - \nabla f(y;\xi)} } \\
        \leq \E{ \sqnorm{ \nabla f(x;\xi) - \nabla f(y;\xi) } } 
        \leq \bar{L}^2 \sqnorm{ x - y } ~.
    \end{align*}
    Hence, taking square roots,
    \begin{align*}
        \|\nabla f(x) - \nabla f(y)\| \leq \bar{L} \|x - y\| ~.
    \end{align*}
\end{proof}
\subsection{Proof of \Cref{lem:expected_similarity}}
\label{proof:expected_similarity}
Again restating the lemma first.
\begin{restate-lemma}{\ref{lem:expected_similarity}}
  Mean-squared smoothness (\Cref{ass:ms_smoothness}) implies expected similarity with the same constant $\bar{L}$:
  \begin{equation*}
    \E{ \sqnorm{\nabla f(x;\xi) - \nabla f(y;\xi) - (\nabla f(x) - \nabla f(y))} }
    \leq \bar{L}^2 \sqnorm{x-y}, \quad \forall\, x,y \in \R^d .
  \end{equation*}
\end{restate-lemma}
\begin{proof}
  This follows by doing some simple algebra
    \begin{align*}
      &\E{\|\nabla f(x;\xi) - \nabla f(y;\xi) - (\nabla f(x) - \nabla f(y))\|^2}\\
      &\quad = \E{\|\nabla f(x;\xi) - \nabla f(y;\xi)\|^2}
        - \|\nabla f(x) - \nabla f(y)\|^2\\
      &\quad \leq \E{\|\nabla f(x;\xi) - \nabla f(y;\xi)\|^2}
      \leq \bar{L}^2 \|x - y\|^2 .
    \end{align*}
\end{proof}
\subsection{Expected Similarity and Bounded Variance for the Minibatch Case}
We now prove another lemma that will be used in the iteration complexity analysis.
\begin{lemma}\label{lem:mb-es-var}
  The following bounds hold:
  \begin{align*}
    \E{ \sqnorm{ \frac{1}{B}\sum_{j=1}^{B}\nabla f(x^{k+1}; \xi^{k,j}) - \nabla f(x^{k+1}) + \nabla f(x^k) - \frac{1}{B}\sum_{j=1}^{B}\nabla f(x^{k}; \xi^{k,j}) } } \\
    \leq \frac{\bar{L}^2}{B}\sqnorm{ x^{k+1}-x^k } ~,
  \end{align*}
  and 
  \begin{align*}
    \E{ \sqnorm{ \frac{1}{B}\sum_{j=1}^{B}\nabla f(x^{k}; \xi^{k,j}) - \nabla f(x^{k}) } }
    \leq \frac{\sigma^2}{B} ~.
  \end{align*}
\end{lemma}
\begin{proof}
  Using \Cref{lem:expected_similarity} and independence of $\xi^{k,j}$:
  \begin{align*}
      & \E{\sqnorm{ \frac{1}{B}\sum_{j=1}^{B}\nabla f(x^{k+1}; \xi^{k,j}) - \nabla f(x^{k+1}) + \nabla f(x^k) - \frac{1}{B}\sum_{j=1}^{B}\nabla f(x^{k}; \xi^{k,j}) }} \\
      &\quad= \frac{1}{B^2} \sum_{j=1}^{B} E\left[\left\|\nabla f(x^{k+1}; \xi^{k,j}) - \nabla f(x^{k+1}) + \nabla f(x^k) - \nabla f(x^{k}; \xi^{k,j})\right\|^2\right] \\
      &\quad\leq \frac{1}{B^2} \sum_{j=1}^{B} \bar{L}^2 \sqnorm{ x^{k+1}-x^k } = \frac{\bar{L}^2}{B} \sqnorm{ x^{k+1}-x^k } ~.
  \end{align*}
  Similarly, using \Cref{ass:unbiased_bounded_variance} and independence:
  \begin{align*}
      \E{ \sqnorm{ \frac{1}{B}\sum_{j=1}^{B}\nabla f(x^{k}; \xi^{k,j}) - \nabla f(x^{k}) } }
      &= \frac{1}{B^2} \sum_{j=1}^{B} \E{ \sqnorm{ \nabla f(x^{k}; \xi^{k,j}) - \nabla f(x^{k}) } } \\
      &\leq \frac{1}{B^2} \sum_{j=1}^{B} \sigma^2 = \frac{\sigma^2}{B} ~.
  \end{align*}
\end{proof}
\section{Why the Current Analysis Cannot Close the Time-Complexity Gap}
\label{sec:why}
Our time lower bound (cf.\ \Cref{thm:lower_bound}) implies, up to universal constants, a requirement of the form
\begin{equation}\label{eq:lb_schematic}
  m_{\mathrm{time}}(B,\Delta,\bar L,\sigma^2)
  \;\ge\;
  \Omega \left(
    \left( \frac{\bar L\Delta}{\sigma\sqrt{\varepsilon}} + 1 \right) 
    T \left(\frac{\sigma^2}{\varepsilon}\right)
  \right),
\end{equation}
where
$$
  T(B) = \min_{m\in[n]}\left( \sum_{i=1}^m\tau_i^{-1} \right)^{-1}(B + m)
$$
is an upper bound on the time needed to obtain $B$ stochastic gradients asynchronously.
On the other hand, the time upper bound we obtain in this paper for \rennalamvr has the schematic structure
\begin{equation}\label{eq:ub_round_based}
  T_{\mathrm{up}}
  \lesssim
  T(B_0) + K\,T(B),
\end{equation}
where \(B_0\) is the initialization batch size, \(B\) is the per-iteration batch size, and \(K\) is the number of server updates (iterations).
If one aims to align \eqref{eq:ub_round_based} with the lower-bound scaling \eqref{eq:lb_schematic} \emph{uniformly over heterogeneous systems} (within our timing model), then the structure of \(T(\cdot)\) can force a specific scaling of \((B,K)\); a simple way to see this is to consider the homogeneous case \(\tau_i\equiv\tau\), for which
\begin{align*}
  T(B) = \min_{m\in[n]}\frac{\tau(B + m)}{m} = \frac{\tau(B + n)}{n} = \frac{\tau}{n}B + \tau ~.
\end{align*}
In particular, even in this ``linear'' regime, \(T(B)\) has a nonzero per-round overhead \(\tau\), and hence
$$
  K\,T(B)=\frac{\tau}{n}KB + \tau K,
$$
while
\begin{align*}
  \left(\frac{\bar L\Delta}{\sigma\sqrt{\varepsilon}}+1\right)\,
  T \left(\frac{\sigma^2}{\varepsilon}\right)
  =
  \left( \frac{\bar L\Delta}{\sigma\sqrt{\varepsilon}} + 1 \right)\left( \frac{\tau}{n}\frac{\sigma^2}{\varepsilon}+\tau \right)
  =
  \frac{\tau}{n}\left(\frac{\bar L\Delta}{\sigma\sqrt{\varepsilon}}+1\right)\frac{\sigma^2}{\varepsilon}
  +\tau\left( \frac{\bar L\Delta}{\sigma\sqrt{\varepsilon}} + 1 \right).
\end{align*}
Thus, matching the lower-bound scaling in this example forces the additive contributions to be of the same order, which yields
\begin{equation}\label{eq:needed_BK}
  K \asymp \frac{\bar L\Delta}{\sigma\sqrt{\varepsilon}} + 1 ~,
\end{equation}
and then matching the leading linear-in-\(B\) terms yields
$
  B \asymp \frac{\sigma^2}{\varepsilon}
$
(with $B_0$ of the same order as well).
However, our iteration-complexity bound \eqref{eq:main_bound} shows that \eqref{eq:needed_BK} cannot be achieved by \rennalamvr under the stepsize restriction \eqref{eq:gamma_choice} used in our analysis.
Let us enforce the batching suggested by the lower bound and set
$
  B_0=B=c\,\frac{\sigma^2}{\varepsilon}
$
for a sufficiently large constant \(c\).
Then \eqref{eq:main_bound} gives
\begin{align*}
  \frac{1}{K}\sum_{k=0}^{K-1}\E{\sqnorm{ \nabla f(x^k) } }
  \le
  \frac{2\Delta}{\gamma K}
    + \frac{2p\sigma^2}{B}
    + \frac{1}{pK}\frac{2\sigma^2}{B_0} ~.
\end{align*}
With \( B=B_0= \nicefrac{c \sigma^2}{\varepsilon} \), the noise terms become
$$
  \frac{2p\sigma^2}{B} = \frac{2p}{c}\,\varepsilon
$$
and
$$
  \frac{1}{pK}\frac{2\sigma^2}{B_0} = \frac{2}{cpK}\,\varepsilon ~.
$$
By choosing \(c\) large enough (and, say, any constant \(p\le 1\)), these two terms can be made \(\le \varepsilon/3\) provided \(K\gtrsim 1/p\).
The remaining requirement comes from the descent term \(\frac{2\Delta}{\gamma K}\).
Crucially, the stepsize constraint \eqref{eq:gamma_choice} enforces \(\gamma = \mathcal{O}(1/\bar L)\) even when \(B\) is as large as \(\sigma^2/\varepsilon\) (the factor \((1-p)/\sqrt{pB}\) only makes \(\gamma\) \emph{smaller}).
Therefore, to make
$$
  \frac{2\Delta}{\gamma K}\le \frac{\varepsilon}{3} ~,
$$
we must have
\begin{equation}\label{eq:K_barLDelta_over_eps}
  K \ge \Omega\!\left(\frac{\bar L\Delta}{\varepsilon}\right).
\end{equation}
In the high-noise regime \(\sigma>\sqrt{\varepsilon}\), the lower-bound-implied target
$$
  K \asymp \frac{\bar L\Delta}{ \sigma\sqrt{\varepsilon} }
$$
is \emph{smaller} than $\nicefrac{\bar L\Delta}{\varepsilon}$ by a factor $\nicefrac{\sigma}{\sqrt{\varepsilon}}$.
Hence, \eqref{eq:K_barLDelta_over_eps} rules out the simultaneous choice \eqref{eq:needed_BK} within our current upper-bound analysis.
In summary, while the lower bound \eqref{eq:lb_schematic} suggests that optimal time would be attained by producing gradients in batches of size \(\sigma^2/\varepsilon\) and performing only
$$
  \Theta \left( \bar L\Delta/(\sigma\sqrt{\varepsilon}) \right)
$$
update rounds, the smoothness-limited descent mechanism behind \eqref{eq:main_bound} forces $\Omega(\nicefrac{\bar L\Delta}{\varepsilon})$ update rounds (under \eqref{eq:gamma_choice}) regardless of how large we take $B$.
This creates a gap relative to \eqref{eq:lb_schematic} for the particular \rennalamvr analysis developed in this paper, which explains why our time upper bound does not match the lower bound in full generality.
Although one can argue that the remaining gap may be due to our proof technique, we believe the bounds are reasonably tight: in particular, when we ignore the timing model and count only oracle queries, our analysis recovers the optimal oracle complexity of \mvr.
\section{Proofs of Upper Bounds}
In this section, we prove the upper bounds on the iteration and time complexity.
\subsection{Proof of \Cref{thm:iteration_complexity}}\label{proof:iteration_complexity}
We start with the iteration complexity.
\begin{restate-theorem}{\ref{thm:iteration_complexity}}[Iteration Complexity of \Cref{algo}]
  Let Assumptions~\ref{ass:lower_bound},~\ref{ass:unbiased_bounded_variance}, and~\ref{ass:ms_smoothness} hold.
  Fix parameters $p\in(0,1]$, integers $B_0\ge 1$ and $B\ge 1$.
  Consider \Cref{algo} with the modification that the initialization $g^0$ uses an unbiased estimate of $\nabla f(x^0)$ using a minibatch of size $B_0$, while for all $k\ge 0$ the iteration-$k$ minibatch has size $B$.
  Choose the stepsize
  \begin{equation}\label{eq:gamma_choice}
    \gamma \le \frac{1}{2\bar{L}\left( 1 + \frac{1-p}{\sqrt{pB}}\right) } ~.
  \end{equation}
  Then for any $K\ge 1$,
  \begin{equation}\label{eq:main_bound}
    \frac{1}{K} \sum_{k=0}^{K-1} \E{ \sqnorm{ \nabla f(x^k)} }
    \le
    \frac{2 \Delta}{\gamma K} + \frac{2 p \sigma^2}{B} + \frac{1}{p K}\frac{2\sigma^2}{B_0},
  \end{equation}
  where $\Delta \coloneqq f(x^0)-f^*$.

  Fix $\varepsilon>0$ and assume $\varepsilon < \sigma^2$ and $\varepsilon < 2\bar{L}\Delta$.
  Choose
  \begin{align}
    \gamma = \frac{1}{4 \bar L},
    \quad p \eqdef \frac{\sqrt{\varepsilon}}{\sigma}, 
    \quad B_0 \eqdef \left\lceil \frac{6\sigma^2}{\varepsilon}\right\rceil, 
    \quad B \eqdef\; \left\lceil \frac{6p\sigma^2}{\varepsilon}\right\rceil= \left\lceil \frac{6\sigma}{\sqrt{\varepsilon}}\right\rceil,
  \end{align}
  then the iterates of \Cref{algo} satisfy
  \begin{align*}
    \frac{1}{K} \sum_{k=0}^{K-1} \E{ \sqnorm{ \nabla f(x^k)} }
    \ \le\ \varepsilon ~,
  \end{align*}
  for
  \begin{align*}
    K \ge \frac{24\Delta\bar L}{\varepsilon}
    + \frac{\sigma}{\sqrt{\varepsilon}} ~.
  \end{align*}
\end{restate-theorem}
\begin{proof}
  Since \Cref{ass:ms_smoothness} implies that $f$ is $\bar L$--smooth (see \Cref{lem:l_smooth}), the standard smoothness inequality (e.g., \citep{nesterov2018lectures}) gives
  \begin{align*}
    f(x^{k+1}) 
    &\leq f(x^k) + \langle \nabla f(x^k), x^{k+1} - x^k \rangle + \frac{\bar{L}}{2} \sqnorm{x^{k+1} - x^k} \\
    &= f(x^k) + \langle \nabla f(x^k), -\gamma g^k \rangle +\frac{\bar L}{2} \sqnorm{x^{k+1} - x^k} \\
    &= f(x^k) + \frac{\gamma}{2} \sqnorm{g^k - \nabla f(x^k)} - \frac{\gamma}{2} \sqnorm{\nabla f(x^k)}
      - \frac{1}{2\gamma}\sqnorm{x^{k+1} - x^k} + \frac{\bar L}{2} \sqnorm{x^{k+1} - x^k} \\
    &= f(x^k) - \frac{\gamma}{2} \sqnorm{\nabla f(x^k)}
      - \left(\frac{1}{2\gamma}-\frac{\bar{L}}{2}\right) \sqnorm{x^{k+1}-x^k}
      + \frac{\gamma}{2} \sqnorm{g^k- \nabla f(x^k)} .
  \end{align*}
  Subtract $f^*$ and take $\ExpCond{\cdot}{\cF_k}$, where $\cF_k \coloneqq \sigma(x^0, \ldots, x^k, g^0, \ldots, g^k)$:
  \begin{align*}
    \E{f(x^{k+1}) - f^* \mid \cF_k}
    &\leq f(x^k) - f^* - \frac{\gamma}{2} \sqnorm{ \nabla f(x^k) } \\
      &\quad - \left(\frac{1}{2\gamma}-\frac{\bar{L}}{2}\right) \sqnorm{ x^{k+1} - x^k }
      + \frac{\gamma}{2} \sqnorm{ g^k- \nabla f(x^k) } .
  \end{align*}
  We need to control the variance term above.
  So, let us study $\ExpCond{\sqnorm{ g^{k+1} - \nabla f(x^{k+1}) }}{\cF_k}$.
  Using bias--variance decomposition:
  \begin{align*}
    &\ExpCond{\sqnorm{ g^{k+1} - \nabla f(x^{k+1}) }}{\cF_k}\\
    &= \ExpCond{\sqnorm{ \frac{1}{B}\sum_{j=1}^{B}\nabla f(x^{k+1}; \xi^{k,j}) - \nabla f(x^{k+1})  + (1-p)\left(g^k  - \frac{1}{B}\sum_{j=1}^{B}\nabla f(x^k; \xi^{k,j})\right) }}{\cF_k}\\
    &= (1-p)^2 \sqnorm{ g^k - \nabla f(x^k) } \\
      &\quad + \ExpCond{\sqnorm{\frac{1}{B}\sum_{j=1}^{B}\nabla f(x^{k+1}; \xi^{k,j}) - \nabla f(x^{k+1})  + (1-p)\left(\nabla f(x^k) - \frac{1}{B}\sum_{j=1}^{B}\nabla f(x^k; \xi^{k,j})\right) }}{\cF_k}\\
    &= (1-p)^2 \sqnorm{ g^k - \nabla f(x^k) } \\
      &\quad + \mathbb{E} \Bigg[ \Bigg\|p\Bigg(\frac{1}{B}\sum_{j=1}^{B}\nabla f(x^{k+1}; \xi^{k,j}) - \nabla f(x^{k+1})\Bigg)\\
      &\quad\quad + (1-p)\Bigg(\frac{1}{B}\sum_{j=1}^{B}\nabla f(x^{k+1}; \xi^{k,j}) - \nabla f(x^{k+1}) + \nabla f(x^k) - \frac{1}{B}\sum_{j=1}^{B}\nabla f(x^k; \xi^{k,j})\Bigg) \Bigg\|^2 \ \Bigg\vert\ \cF_k \Bigg]\\
    &\leq (1-p)^2 \sqnorm{ g^k - \nabla f(x^k) }
      + 2 p^2 \ExpCond{\sqnorm{\frac{1}{B}\sum_{j=1}^{B}\nabla f(x^{k+1}; \xi^{k,j}) - \nabla f(x^{k+1})}}{\cF_k} \\
      &\quad + 2(1-p)^2 \ExpCond{ \sqnorm{\frac{1}{B}\sum_{j=1}^{B}\nabla f(x^{k+1}; \xi^{k,j}) - \nabla f(x^{k+1}) + \nabla f(x^k) - \frac{1}{B}\sum_{j=1}^{B}\nabla f(x^k; \xi^{k,j})} }{\cF_k} .
  \end{align*}
  Using \cref{lem:mb-es-var}, we get
  \begin{equation*}
    \ExpCond{\sqnorm{ g^{k+1} - \nabla f(x^{k+1}) }}{\cF_k}
    \leq (1-p)^2 \sqnorm{ g^k - \nabla f(x^k) }
    + \frac{2 p^2 \sigma^2}{B}
    + \frac{2(1-p)^2 \bar{L}^2}{B} \sqnorm{ x^{k+1}-x^k } .
  \end{equation*}
  Define the Lyapunov function:
  \begin{equation*}
    \Psi_k = f(x^k) - f^* + \frac{\gamma}{2p} \sqnorm{ g^k - \nabla f(x^k) } .
  \end{equation*}
  We have:
  \begin{align*}
    \ExpCond{\Psi_{k+1}}{\cF_k}
    &\leq f(x^k) - f^* - \frac{\gamma}{2} \sqnorm{ \nabla f(x^k) }
    - \left(\frac{1}{2\gamma}-\frac{\bar{L}}{2}\right) \sqnorm{ x^{k+1} - x^k }
    + \frac{\gamma}{2} \sqnorm{ g^k- \nabla f(x^k) } \\
      &\quad + \frac{\gamma}{2p} \left((1-p)^2 \sqnorm{ g^k - \nabla f(x^k) }
      + \frac{2 p^2 \sigma^2}{B}
      + \frac{2(1-p)^2 \bar{L}^2}{B} \sqnorm{ x^{k+1}-x^k }\right) \\
    &= f(x^k) - f^* - \frac{\gamma}{2} \sqnorm{ \nabla f(x^k) } - \left(\frac{1}{2\gamma}-\frac{\bar{L}}{2} - \frac{\gamma(1-p)^2 \bar{L}^2}{p B}\right) \sqnorm{ x^{k+1}-x^k } \\
    &\quad + \frac{\gamma}{2p}\left(p + (1-p)^2\right) \sqnorm{ g^k - \nabla f(x^k) }
    + \frac{\gamma p \sigma^2}{B} ~.
  \end{align*}
  Using $p + (1-p)^2 \le 1$, we get
  \begin{equation*}
    \ExpCond{\Psi_{k+1}}{\cF_k}
    \leq \Psi_k - \frac{\gamma}{2} \sqnorm{ \nabla f(x^k) }
    - \left(\frac{1}{2\gamma}-\frac{\bar{L}}{2} - \frac{\gamma(1-p)^2 \bar{L}^2}{p B}\right) \sqnorm{ x^{k+1}-x^k }
    + \frac{\gamma p \sigma^2}{B} ~.
  \end{equation*}
  Next to ensure that the coefficient before $\sqnorm{ x^{k+1}-x^k }$ is non-negative, it is sufficient to have
  \begin{equation*}
    \frac{1}{4\gamma} \geq \frac{\bar{L}}{2} \quad\text{and}\quad \frac{1}{4\gamma} \geq \frac{\gamma(1-p)^2 \bar{L}^2}{p B} ~,
  \end{equation*}
  which is equivalent to
  \begin{equation*}
    \gamma \leq \frac{1}{2\bar{L}} \quad\text{and}\quad \gamma \leq \frac{1}{2\bar{L} \frac{1-p}{\sqrt{p B}}} ~.
  \end{equation*}
  Hence, taking $\gamma \le \frac{1}{2\bar{L}\left( 1 + \frac{1-p}{\sqrt{pB}}\right) }$ ensures
  $\frac{1}{2\gamma} - \frac{\bar{L}}{2} - \frac{\gamma(1-p)^2\bar{L}^2}{pB} \ge 0$.
  Thus, we obtained
  \begin{align*}
    \ExpCond{\Psi_{k+1}}{\cF_k}
    &\leq \Psi_k - \frac{\gamma}{2} \sqnorm{ \nabla f(x^k) }
    + \frac{\gamma p \sigma^2}{B} ~.
  \end{align*}
  Taking full expectation and telescoping from $k=0$ to $K-1$ yields
  \begin{align*}
    \sum_{k=0}^{K-1} \left(\frac{\gamma}{2}\E{\sqnorm{ \nabla f(x^k) }} 
      - \frac{\gamma p\sigma^2}{B}\right)
    \le \Delta + \frac{\gamma}{2p} \E{\sqnorm{ g^0 - \nabla f(x^0) }} .
  \end{align*}
  Hence,
  \begin{align*}
    \frac{1}{K} \sum_{k=0}^{K-1} \E{\sqnorm{ \nabla f(x^k) }}
    \;\le\; \frac{2 \Delta}{\gamma K} 
      \;+\; \frac{2 p \sigma^2}{B} 
      \;+\; \frac{1}{p K} \frac{2\sigma^2}{B_0} ~.
  \end{align*}
  Fix $\varepsilon > 0$.
  We assume $\varepsilon < \sigma^2$.
  Otherwise, we are in the low-noise regime, hence we can set $p=1$ and use \sgd with minibatch size $B$, which reaches the target accuracy in $\cO\!\left(\frac{\bar{L} \Delta}{\varepsilon}\right)$ stochastic gradients.
  Also assume $\varepsilon < 2\bar{L}\Delta$.
  Otherwise, since $\|\nabla f(x^0)\|^2\leq 2\bar{L}\Delta \leq \varepsilon$, the initial point $x^0$ is already $\varepsilon$--stationary.

  Recall that we have established
  \begin{align*}
    \frac{1}{K} \sum_{k=0}^{K-1} \E{\sqnorm{ \nabla f(x^k) }}
    &\le \frac{2 \Delta }{\gamma K}
      \;+\; \frac{2 p \sigma^2}{B} 
      \;+\; \frac{1}{p K} \frac{2\sigma^2}{B_0} ~.
  \end{align*}
  To make sure the right hand side is bounded by $\varepsilon$, we bound each term on the right-hand side by $\nicefrac{\varepsilon}{3}$.

  We choose the initialization batch size
  \begin{equation}\label{eq:choice_S0}
    B_0 \eqdef \left\lceil \frac{6\sigma^2}{\varepsilon}\right\rceil,
  \end{equation}
  which gives
  \begin{align*}
    \frac{1}{pK}\frac{2\sigma^2}{B_0}
    \;\le\;
    \frac{1}{pK}\cdot \frac{\varepsilon}{3} ~.
  \end{align*}
  To ensure this term is at most $\nicefrac{\varepsilon}{3}$, we require
  \begin{equation}\label{eq:K_ge_1_over_p}
    K \;\ge\; \frac{1}{p} ~.
  \end{equation}

  We set
  \begin{equation}\label{eq:choice_S}
    B \eqdef \left\lceil \frac{6p\sigma^2}{\varepsilon}\right\rceil,
  \end{equation}
  which directly ensures
  \begin{align*}
      \frac{2p\sigma^2}{B} \;\le\; \frac{\varepsilon}{3} ~.
  \end{align*}
  
  We choose
  \begin{equation}\label{eq:choice_p}
    p \eqdef \frac{\sqrt{\varepsilon}}{\sigma} ~,
  \end{equation}
  Substituting \eqref{eq:choice_p} into \eqref{eq:choice_S} yields
  \begin{align}\label{eq:B_simplified}
    B = \left\lceil \frac{6\sigma}{\sqrt{\varepsilon}}\right\rceil.
  \end{align}
  Moreover, this choice of $p$ gives
  \begin{align*}
    \frac{1-p}{\sqrt{pB}}
    &\le \frac{1-p}{\sqrt{p\cdot 6p\sigma^2/\varepsilon}}
    = \frac{1}{\sqrt{6}}\cdot \frac{1-p}{p}\cdot \frac{\sqrt{\varepsilon}}{\sigma}\\
    &= \frac{1}{\sqrt{6}}(1-p) 
    \;\le\; \frac{1}{\sqrt{6}} ~.
  \end{align*}

  The above bound ensures that the step size
  \begin{equation}\label{eq:choice_gamma}
    \gamma \eqdef \frac{1}{4\bar{L}}
  \end{equation}
  satisfies the required condition, since
  \begin{align*}
    2\bar{L}\!\left(1+ \frac{1-p}{\sqrt{pB}}\right) 
    \;\le\; 2\bar{L}\!\left(1+ \frac{1}{\sqrt{6}}\right)
    \;<\; 4\bar{L} ~.
  \end{align*}

  With $\gamma = \nicefrac{1}{4\bar{L}}$, bounding the first term by $\nicefrac{\varepsilon}{3}$ requires
  \begin{equation}\label{eq:K_ge_main}
    K \;\ge\; \frac{24\Delta\bar{L}}{\varepsilon} ~.
  \end{equation}
  Combining with \eqref{eq:K_ge_1_over_p} and recalling $p = \nicefrac{\sqrt{\varepsilon}}{\sigma}$, we take
  \begin{equation}\label{eq:choice_K}
    K \;\ge\; \frac{24\Delta\bar{L}}{\varepsilon}
      + \frac{\sigma}{\sqrt{\varepsilon}} ~.
  \end{equation}

  The parameter choices \eqref{eq:choice_S0}, \eqref{eq:B_simplified}, \eqref{eq:choice_p}, \eqref{eq:choice_gamma}, and \eqref{eq:choice_K} together guarantee
  \begin{align*}
    \frac{1}{K}\sum_{k=0}^{K-1}\E{\sqnorm{ \nabla f(x^k) }} \;\le\; \varepsilon ~.
  \end{align*}
\end{proof}
\subsection{Proof of \Cref{thm:time_complexity}}
\label{proof:time_complexity}
Before proving the theorem, we first prove a simple lemma.
A result of this type was used in \citep[Theorem~7.5]{tyurin2023optimal}; here we provide a simpler proof.
\begin{lemma}\label{lem:time_to_collect_B}
  Suppose we have $n$ workers with computation times as in \Cref{ass:fixed_time}.
  Let $B\in\{1,2,\ldots\}$.
  Consider any iteration boundary at which the server starts collecting new stochastic gradients.
  Then, the time needed to collect $B$ stochastic gradients is at most
  \begin{equation*}
    2 T(B) = 2 \min_{m\in[n]}\left(\sum_{i=1}^{m}\frac{1}{\tau_i}\right)^{-1}(B + m) ~.
  \end{equation*}
\end{lemma}
\begin{proof}
  Fix $t>0$. For worker $i$, in the worst case the worker finishes and sends a gradient just before the iteration starts.
  Therefore, the first interval of length $\tau_i$ after the iteration starts may not produce a \emph{new} gradient usable in this iteration.
  Hence, within $t$ seconds from the iteration start, worker $i$ can contribute at least
  \begin{equation*}
    n_i(t) \coloneqq \max\left\{\left\lfloor \frac{t-\tau_i}{\tau_i}\right\rfloor,0\right\}
    = \max\left\{\left\lfloor \frac{t}{\tau_i}\right\rfloor-1,0\right\}
  \end{equation*}
  new gradients.
  Set
  \begin{align*}
    j^* = \argmin_{j\in[n]}\left(\sum_{i=1}^{j}\frac{1}{\tau_i}\right)^{-1}(B + j) ~,
  \end{align*} 
  then, since $\max\{u,0\}\ge u$, we have $n_i(t)\ge \left\lfloor \frac{t}{\tau_i}\right\rfloor-1$, and therefore
  \begin{align*}
    \sum_{i=1}^{n} n_i(t) \ge \sum_{i=1}^{j^*} n_i(t)
    \;\ge\; \sum_{i=1}^{j^*}\left(\left\lfloor \frac{t}{\tau_i}\right\rfloor-1\right)
    \;\ge\; \sum_{i=1}^{j^*}\left(\frac{t}{\tau_i}-2\right)
    \;=\; t\sum_{i=1}^{j^*}\frac{1}{\tau_i}-2j^*,
  \end{align*}
  where we used $\lfloor a\rfloor \ge a-1$ for all $a\ge 0$.
  Thus, if $t$ satisfies
  \begin{equation*}
    t\sum_{i=1}^{j^*}\frac{1}{\tau_i}-2j^* \;\ge\; B,
  \end{equation*}
  then by time $t$ the server can collect at least $B$ gradients from workers $1,\ldots,j^*$.
  Choosing
  \begin{equation*}
    t=2 \left(\sum_{i=1}^{j^*}\frac{1}{\tau_i}\right)^{-1}(B+j^*)
  \end{equation*}
  makes the left-hand side equal to $2B \geq B$, concluding the proof.
\end{proof}
\begin{restate-theorem}{\ref{thm:time_complexity}}[Time Complexity of \Cref{algo}]
  Assume the setup of \Cref{thm:iteration_complexity} and our distributed time model:
  there are $n$ workers with per-sample computation times $0<\tau_1\le \cdots \le \tau_n$.
  Assume that on worker $i$, computing the pair $\bigl(\nabla f(x^{k}; \xi),\nabla f(x^{k+1}; \xi)\bigr)$ takes $2\tau_i$ seconds.

  Then, the initialization (collecting $B_0$ single gradients) takes at most
  \begin{equation}\label{eq:time_init_mvr}
    2 \cdot \min_{m\in[n]}\left(\sum_{i=1}^{m}\frac{1}{\tau_i}\right)^{-1}(B_0 + m)
  \end{equation}
  seconds, and each iteration (collecting $B$ gradient-pairs) takes at most
  \begin{equation}\label{eq:time_iter_mvr}
    4 \cdot \min_{m\in[n]}\left(\sum_{i=1}^{m}\frac{1}{\tau_i}\right)^{-1}(B + m)
  \end{equation}
  seconds.

  In particular, with the parameter choice of \Cref{thm:iteration_complexity}, \Cref{algo} returns an $\varepsilon$--stationary point within
  \begin{equation}\label{eq:time_mvr_eps}
    T_{\mathrm{MVR}}
    \;\le\;
    2 \cdot \min_{m\in[n]}\left(\sum_{i=1}^{m}\frac{1}{\tau_i}\right)^{-1}(B_0 + m)
    \;+\;
    4K\cdot \min_{m\in[n]}\left(\sum_{i=1}^{m}\frac{1}{\tau_i}\right)^{-1}(B + m)
  \end{equation}
  seconds, where
  \begin{align*}
    B_0 &= \left\lceil \frac{6\sigma^2}{\varepsilon}\right\rceil,
    &
    B &= \left\lceil \frac{6\sigma}{\sqrt{\varepsilon}}\right\rceil,
    &
    K &\ge \frac{24\Delta\bar L}{\varepsilon}
    + \frac{\sigma}{\sqrt{\varepsilon}} ~.
  \end{align*}
  Thus,
  \begin{equation*}
    T_{\mvr}
    = \cO\Bigg( \left( \frac{\sigma}{\sqrt{\varepsilon}} + \frac{ \bar{L} \Delta }{\varepsilon} \right) \min_{m\in[n]} \left(\sum_{i=1}^{m}\frac{1}{\tau_i}\right)^{-1} \left( \frac{\sigma}{\sqrt{\varepsilon}} + m \right) 
    + \min_{m\in[n]} \left(\sum_{i=1}^{m}\frac{1}{\tau_i}\right)^{-1}\left( \frac{\sigma^2}{\varepsilon} + m \right) \Bigg) ~.
  \end{equation*}
\end{restate-theorem}
\begin{proof}
  The initialization bound follows from \Cref{lem:time_to_collect_B} applied with $B=B_0$ and computation times $\tau_i$.

  For each iteration, the server collects $B$ arrivals, and each arrival produced by worker $i$ is a gradient-pair and takes $2\tau_i$ seconds.
  Applying \Cref{lem:time_to_collect_B} with computation times $2\tau_i$ gives that one iteration takes at most
  \begin{align*}
    2\min_{m\in[n]}\left(\sum_{i=1}^{m}\frac{1}{2\tau_i}\right)^{-1}(B+m)
    =
    4 \min_{m\in[n]}\left(\sum_{i=1}^{m}\frac{1}{\tau_i}\right)^{-1}(B+m).
  \end{align*}
  Summing the initialization time and $K$ iteration times yields the desired bound.
\end{proof}
\section{Proofs of Lower Bound}
We start by introducing several definitions and notations that will be used throughout this section.

For a vector $x\in\R^d$, let $\operatorname{support}(x)\coloneqq\{i\in[d]\mid x_i\neq 0\}$ and $x_{\ge i}\coloneqq(x_i,\ldots,x_d)\in\R^{d-i+1}$.
For $\alpha\in[0,1)$ define the \emph{progress}
\begin{equation}\label{eq:progress}
  \operatorname{prog}_{\alpha}(x)\coloneqq\max\{\,i\in\{0,1,\dots,d\}\mid |x_i|>\alpha\,\},
  \qquad\text{with }~ x_0\equiv 1 ~.
\end{equation}

Our goal is to construct functions and stochastic oracles such that, under zero-respecting algorithms, each completed oracle reply can activate at most one new coordinate.

\begin{definition}[First-order zero-chain]\label{def:zero-chain}
A differentiable function $F:\R^T\to\R$ is a (first-order) zero-chain if, for all $x\in\R^T$,
\begin{equation}\label{eq:zc-deterministic}
  \operatorname{prog}_{0} \bigl(\nabla F(x)\bigr)\ \le\ \operatorname{prog}_{0}(x)+1.
\end{equation}
\end{definition}

\paragraph{Intuition.}
In the noiseless case $g(x,\xi)\equiv\nabla F(x)$, \eqref{eq:zc-deterministic} implies that a zero-respecting
algorithm can reveal at most one new coordinate per completed oracle reply, so its progress is at most linear
in the number of replies.

\begin{definition}[Probability-$p$ zero-chain]\label{def:pzc}
  A stochastic mapping $g:\R^T\times \cD \to\R^T$ is a probability-$p$ zero-chain if, for all $x\in\R^T$,
  \begin{align}
    \mathbb{P}_\xi\!\left(\operatorname{prog}_0 \bigl(g(x,\xi)\bigr) = \operatorname{prog}_{\tfrac14}(x)+1\right) &\le p, \label{eq:pzc-1}\\
    \mathbb{P}_\xi\!\left(\operatorname{prog}_0 \bigl(g(x,\xi)\bigr) > \operatorname{prog}_{\tfrac14}(x)+1\right) &= 0. \label{eq:pzc-2}
  \end{align}
\end{definition}

\subsection*{Setup}
We use the same deterministic chain function as in the works by \citet{tyurin2023optimal} and \citet{arjevani2022lower}.
Define $F_T:\R^T\!\to\R$ by
\begin{equation*}
  F_T(x)\ \coloneqq-\Psi(1)\,\Phi(x_1)+\sum_{i=2}^T \Bigl(\Psi(-x_{i-1})\,\Phi(-x_i)\ -\ \Psi(x_{i-1})\,\Phi(x_i)\Bigr),
\end{equation*}
where
\begin{align*}
  &\Psi(t)=
  \begin{cases}
    0, & t\le \tfrac12,\\[3pt]
    \exp\!\Bigl(1-\frac{1}{(2t-1)^2}\Bigr), & t>\tfrac12,
  \end{cases}\\
  \qquad
  &\Phi(t)\ =\ \sqrt{e}\ \int_{-\infty}^{t} e^{-\tfrac12 \tau^2}\, d\tau.
\end{align*}
\begin{lemma}[Properties of $F_T$, cf.\ Lemma~2]\label{lem:FT-props}
  There are absolute constants $\Delta_0=12$, $\ell_1=152$, $\gamma_\infty=23$ such that:
  \begin{enumerate}
  \item $F_T(0)-\inf_x F_T(x)\le \Delta_0\cdot T$.
  \item $\nabla F_T$ is $\ell_1$–Lipschitz in $\ell_1$.
  \item For all $x$, $\|\nabla F_T(x)\|_\infty\le \gamma_\infty$.
  \item $\operatorname{prog}_0\!\bigl(\nabla F_T(x)\bigr)\ \le\ \operatorname{prog}_{1/2}(x)+1$.
  \item If $\operatorname{prog}_{1}(x)<T$, then $\|\nabla F_T(x)\|\ \ge\ \bigl|\nabla_{\operatorname{prog}_0(x)+1}F_T(x)\bigr|\ >\ 1$.
  \end{enumerate}
\end{lemma}

Next we define the estimator. Let $\Gamma:\R\to\R$ be smooth, non-decreasing, and Lipschitz, with
\begin{equation}\label{eq:Gamma-def}
  \Gamma(t)=0\ \text{ for } t\le \tfrac14, \qquad \Gamma(t)=1\ \text{ for } t\ge \tfrac12.
\end{equation}
Define, for each $i$,
\begin{align}
  \Theta_i(x)\coloneqq\Gamma\!\left( 1 - \biggl\|\Gamma\bigl(|x_{\ge i}|\bigr)\biggr\|_2 \right) \label{eq:Theta-def}
\end{align}
so that
\begin{align}\label{eq:Theta-bounds}
  \mathbf{1}\{i>\operatorname{prog}_{1/4}(x)\}\ \le\ \Theta_i(x)\ \le\ \mathbf{1}\{i>\operatorname{prog}_{1/2}(x)\}.
\end{align}
A concrete choice is obtained by the “integrated bump”:
\begin{align}\label{eq:Gamma-bump}
  \Gamma(t)=\frac{\int_{1/4}^{t}\Lambda(\tau)\,d\tau}{\int_{1/4}^{1/2}\Lambda(\tau')\,d\tau'},
  \quad \Lambda(t)= 
  \begin{cases}
    0, & t\le \tfrac14\ \text{or}\ t\ge \tfrac12,\\[4pt]
    \exp\!\Bigl( -\dfrac{1}{100\,(t-\tfrac14)\,(\tfrac12-t)} \Bigr), & \tfrac14<t<\tfrac12.
  \end{cases}
\end{align}
This $\Gamma$ satisfies: $\Gamma\in C^\infty$, $0\le \Gamma'(t)\le 6$, and $|\Gamma''(t)|\le 128$.

Define the smoothed estimator
\begin{align}\label{eq:gbar}
  \bigl[\bar g_T(x,\xi)\bigr]_i\ \coloneqq\ \nabla_i F_T(x)\,\nu_i(x,\xi), \qquad 
  \nu_i(x,\xi)\coloneqq1+\Theta_i(x)\Bigl(\frac{\xi}{p}-1\Bigr),\quad \xi\sim\mathrm{Bernoulli}(p).
\end{align}
\begin{lemma}[Mean-squared smooth estimator (Lemma~4 in \citet{arjevani2022lower})]\label{lem:gbar}
  $\bar g_T$ is a probability-$p$ zero-chain, is unbiased for $\nabla F_T$, and there exist constants $\varsigma=23$ and $\bar\ell_1=328$ such that for all $x,y\in\R^T$,
  \begin{align*}
    \E{\bigl\|\bar g_T(x,z)-\nabla F_T(x)\bigr\|^2} &\le \varsigma^2\,\frac{1-p}{p}, \\
    \E{\bigl\|\bar g_T(x,z)-\bar g_T(y,z)\bigr\|^2} &\le \frac{\bar\ell_1^{\,2}}{p}\,\|x-y\|^2.
  \end{align*}
\end{lemma}
\subsection{Proof of \Cref{thm:lower_bound}}
\label{proof:lower_bound}
Let us first restate the theorem before proving it.
\begin{restate-theorem}{\ref{thm:lower_bound}}
  Fix $\Delta>0$, $\bar{L}>0$, $\sigma^2>0$, $0<\varepsilon<c'\bar{L}\Delta$, an integer $B\ge1$, and $n$ workers with batch-time functions $\{\tau_i(\cdot)\}_{i=1}^n$ that satisfy \Cref{ass:batch-time}.
  Write $\tau_i\coloneqq\tau_i(1)$ and assume $0<\tau_1\le\cdots\le\tau_n$.
  Let the chain constants $\Delta_0,\ell_1,\gamma_\infty$ be as in \Cref{lem:FT-props} and the estimator constants $\varsigma,\bar\ell_1$ as in \Cref{lem:gbar}.
  Define
  \begin{equation*}
    p\;\eqdef\;\min\!\left\{\frac{2\varepsilon \varsigma^2 }{\sigma^2},\,1\right\},
    \qquad
    L\;\eqdef\;\frac{\ell_1}{\bar\ell_1}\,\bar L\,\sqrt{p}\ \ (\le \bar L),
  \end{equation*}
  \begin{equation*}
    \lambda\;\eqdef\;\frac{\ell_1}{L}\,\sqrt{2\varepsilon},
    \qquad
    T\;\eqdef\;\Biggl\lfloor \frac{L\Delta}{2\,\Delta_0\,\ell_1\,\varepsilon}\Biggr\rfloor.
  \end{equation*}
  Then there exist $f\in\cF_{\Delta,\bar L}$ and an oracle class $\mathcal{O}\in \cO_{\tau_1(\cdot),\dots,\tau_n(\cdot)}^{\sigma^2,\bar L, B}(f)$ such that, under Protocol~\ref{protocol},
  \begin{align*}
    m_{\mathrm{time}}(B, \Delta, \bar L, \sigma^2)
    \ \ge\ c \cdot \left( \frac{\bar L \Delta \min\left\{\nicefrac{\sqrt{\varepsilon}}{\sigma},\,1\right\} }{\varepsilon} + 1 \right) \
    \min_{m\in[n]}
    \left(\sum_{i=1}^{m}\frac{1}{\tau_i}\right)^{-1}
    \left( \frac{\sigma^2}{\varepsilon} + m \right) .
  \end{align*}
\end{restate-theorem}
\begin{proof}
  Before presenting the proof, we briefly outline the argument.
  We first follow the construction by \citet{arjevani2022lower}, and then derive the time complexity lower bound using the time protocol analysis from \citet{tyurin2023optimal}.

We work under Protocol~\ref{protocol} with batch sizes bounded by $B$.
By the definition of zero-respecting algorithms, for any interaction $r\ge1$ and any $k\in[B]$,
\begin{equation*}
  \operatorname{support}\!\left(x^{(r,k)}\right)\ \subseteq\ \bigcup_{s<r}\ \bigcup_{k'\in[B]}\operatorname{support}\!\left(g^{(s,k')}\right).
\end{equation*}
At $r=1$ the union on the right is empty, hence $x^{(1,k)}=0$ for all $k\in[B]$.

  Fix parameters $\Delta>0$, accuracy $\varepsilon>0$, and let $L\le\bar L$ be chosen below. Define
  \begin{equation}\label{eq:f-rescaled}
    f(x)=\frac{L\lambda^{2}}{\ell_1}\,F_T\!\left(\frac{x}{\lambda}\right),
    \qquad
    \lambda=\frac{\ell_1}{L}\,\sqrt{2\varepsilon},
    \qquad
    T=\left\lfloor \frac{\Delta}{\Delta_0\,(L\lambda^2/\ell_1)}\right\rfloor
    =\left\lfloor \frac{L\Delta}{2\Delta_0\ell_1\,\varepsilon}\right\rfloor.
  \end{equation}
  By \Cref{lem:FT-props} (i),(ii), \(f\in\cF_{\Delta,\bar L}\).
  Moreover, \(\nabla f(x)=\frac{L\lambda}{\ell_1}\,\nabla F_T(x/\lambda)\).

  Let the Bernoulli parameter \(p\in(0,1]\) be chosen below.
  Define
  \begin{align}\label{eq:g-rescaled}
    \nonumber
    \bigl[\nabla f(x;\xi)\bigr]_j
    =\frac{L\lambda}{\ell_1}\,\bigl[\bar g_T(x/\lambda,\xi)\bigr]_j
    &=\frac{L\lambda}{\ell_1}\,\nabla_j F_T\!\left(\frac{x}{\lambda}\right)\!
    \Bigl(1+\Theta_j(x/\lambda)\bigl(\tfrac{\xi}{p}-1\bigr)\Bigr)\\
    &= \nabla_j f(x)\Bigl(1+\Theta_j(x/\lambda)\bigl(\tfrac{\xi}{p}-1\bigr)\Bigr),
    \quad \xi\sim{\rm Bernoulli}(p).
  \end{align}

  By \Cref{lem:gbar}, we have
  \[
    \mathbb E[\nabla f(x;\xi)]= \nabla f(x),
  \]
  and
  \[
    \mathbb E\bigl\|\nabla f(x;\xi)-\nabla f(x)\bigr\|^2
    \le \Bigl(\tfrac{L\lambda}{\ell_1}\Bigr)^2\,\varsigma^2\,\frac{1-p}{p}.
  \]
  Choosing
  \begin{equation}\label{eq:p-choice}
    p=\min\!\left\{\frac{2\varepsilon \varsigma^2 }{\sigma^2},\,1\right\}
  \end{equation}
  makes the variance \(\le\sigma^2\).

  Again by Lemma~\ref{lem:gbar},
  \[
    \mathbb E\bigl\|\nabla f(x;\xi)-\nabla f(y;\xi)\bigr\|^2
    \le \Bigl(\tfrac{\bar\ell_1 L}{\ell_1\sqrt p}\Bigr)^2\,\|x-y\|^2.
  \]
  Thus taking
  \begin{equation}\label{eq:L-choice}
    L=\frac{\ell_1}{\bar\ell_1}\,\bar L\,\sqrt p
    =\frac{\ell_1}{\bar\ell_1}\,\bar L\cdot \min\left\{\frac{\varsigma \sqrt{2\varepsilon}}{\sigma},\,1\right\}
    \ \le\ \bar L
  \end{equation}
  ensures that $\nabla f$ belongs to the oracle class in Definition~\ref{def:oracle_class}.

  We now derive the time lower bound following \citet{tyurin2023optimal}.
  From \Cref{lem:FT-props} (v) we have that if $\progg_{1}(u)<T$, then $\|\nabla F_T(u)\| > 1$.
  Using the monotonicity of $\progg_\alpha$ in $\alpha$ and $\progg_0(x/\lambda)=\progg_0(x)$, we obtain, for any $x\in\R^T$,
  \begin{align*}
    \progg_0(x) < T
    \;\Longrightarrow\;
    \|\nabla f(x)\|^2
    = \left\|\frac{L\lambda}{\ell_1}\,\nabla F_T(x/\lambda)\right\|^2
    = 2\varepsilon \left\|\nabla F_T(x/\lambda)\right\|^2
    > 2\varepsilon.
  \end{align*}
  Equivalently,
  \begin{equation}
    \label{eq:grad-lb-mss}
    \|\nabla f(x)\|^2 > 2\varepsilon\,\mathbf{1}\!\left\{\progg_0(x)<T\right\}
    \qquad\text{for all }x\in\R^T.
  \end{equation}
  From \Cref{lem:FT-props}(iv), we also have
  \begin{align*}
    \progg_{0}(\nabla f(x))
    &= \progg_{0}\!\left(\frac{L\lambda}{\ell_1}\,\nabla F_T(x/\lambda)\right)
    = \progg_{0}\bigl(\nabla F_T(x/\lambda)\bigr) \\
    &\leq \progg_{1/2}(x/\lambda) + 1
    \leq \progg_{1/4}(x/\lambda) + 1.
  \end{align*}
  Thus, for indices $i>\progg_{1/4}(x/\lambda)+1$ we have $\nabla_i f(x)=0$, and using~\eqref{eq:g-rescaled},
  \begin{align*}
    [\nabla f(x;\xi)]_i = 0,
    \qquad \text{if } i> \progg_{1/4}(x/\lambda)+1.
  \end{align*}
  Moreover, by~\eqref{eq:Theta-bounds}, $\Theta_i(x/\lambda)=1$ for $i = \progg_{1/4}(x/\lambda)+1$, and hence
  \begin{align*}
    [\nabla f(x;\xi)]_i
    = \nabla_i f(x)\,\frac{\xi}{p},
    \qquad \text{if } i = \progg_{1/4}(x/\lambda)+1.
  \end{align*}
  Since $\progg_{1/4}(x/\lambda)\le\progg_0(x)$, the only way to activate a new coordinate
  $i>\progg_0(x)$ is to have $\xi=1$, and this can happen only at $i=\progg_{1/4}(x/\lambda)+1\le \progg_0(x)+1$.
  Therefore, $\nabla f$ is a probability-$p$ zero-chain in the sense of Definition~\ref{def:pzc}.

 Now let $\{x^{(r,k)}\}_{r\ge1,k\in[B]}$ be the query points of a zero-respecting algorithm
$\mathsf{A}\in\cA_{\mathrm{zr}}$ interacting with the oracles under Protocol~\ref{protocol}, and let
$g^{(r,k)} = \nabla f(x^{(r,k)};\xi^{(r)})$ denote the completed stochastic gradients associated with interaction $r$.

  By the zero-respecting property and the probability-$p$ zero-chain structure, the process
  $\max_{s\le r,\, k\in[B]}\progg_0(x^{(s,k)})$ can increase by at most one per completed oracle reply, and this increase
  occurs only when $\xi^{(r)}=1$, which happens with probability at most $p$ (the same $\xi^{(r)}$ is used for all $k\in[B]$).
  Moreover, since $\tau_i(1)\le \tau_i(k)$ for any batch size $k\le B$, replacing the batch-time function by the smaller delay $\tau_i(1)$ can only make the oracles faster. Therefore, any lower bound proved with delays $\tau_i=\tau_i(1)$ also applies to the original model.

  Therefore, all assumptions of Lemma~D.2 in \citet{tyurin2023optimal} are satisfied (with $\textnormal{prog}(x)\equiv\progg_0(x)$ and delays $\tau_i$), and we can invoke it directly. In particular, for any $\delta\in(0,1)$ and any time $t$ satisfying
  \begin{equation}\label{eq:lb-time-raw}
    t\ 
    \le\ \frac{1}{24}\ \left(\frac{T}{2}+\log\frac{1}{\delta}\right) \ 
    \min_{m\in[n]}
    \left(\sum_{i=1}^{m}\frac{1}{\tau_i}\right)^{-1}
    \left(\frac{1}{p}+m\right) ,
  \end{equation}
  we have, with probability at least $1-\delta$, that $\progg_0(x^{(r,k)})<T$ for all queried points $(r,k)$ whose replies have been received by time $t$.
  Combining this with~\eqref{eq:grad-lb-mss} yields
  \[
    \inf_{(r,k)\in S_t}\|\nabla f(x^{(r,k)})\|^2 > 2\varepsilon
    \quad\text{with probability at least }1-\delta,
  \]
  and hence
  \[
    \E{\ \inf_{(r,k)\in S_t}\ \|\nabla f(x^{(r,k)})\|^2\ }
    \ge 2\varepsilon(1-\delta).
  \]
  Choosing $\delta=\tfrac12$ gives
  \[
    \E{ \inf_{(r,k)\in S_t}\ \|\nabla f(x^{(r,k)})\|^2\ }
    > \varepsilon
  \]
  whenever~\eqref{eq:lb-time-raw} holds with $\delta=\tfrac12$.

  Finally, substituting $T$ from \eqref{eq:f-rescaled} and recalling the choice of $L$ in \eqref{eq:L-choice}, we obtain that it is necessary to have
  \begin{align} \label{eq:time-lb-intermediate}
    t\ \geq \frac{1}{24}\ \min_{m\in[n]}
    \left[\left(\sum_{i=1}^{m}\frac{1}{\tau_i}\right)^{-1}
    \left(\frac{1}{p} +m\right)\right]\,
    \left(\frac{1}{2}\left\lfloor
    \frac{\bar L \Delta \sqrt{p}}{2 \Delta_0 \bar \ell_1 \varepsilon}
    \right\rfloor-1\right),
  \end{align}
  to ensure
  \[
    \mathbb{E}\!\left[\ \inf_{(r,k)\in S_t}\ \|\nabla f(x^{(r,k)})\|^2\ \right]
    \leq \varepsilon.
  \]
  As in \citet{arjevani2022lower} we consider the cases
  $
  \frac{\bar L \Delta\sqrt{p}}{2 \Delta_0 \bar \ell_1 \varepsilon} \geq 5
  $
  and
  $
  \frac{\bar L \Delta\sqrt{p}}{2 \Delta_0 \bar \ell_1 \varepsilon} < 5
  $
  separately.
  In the former case, using $\lfloor x \rfloor /2 - 1 \geq x/4$, we get
  \begin{align*}
    t\ 
    &\geq \frac{1}{24}\ \frac{\bar L \Delta \sqrt{p}}{8  \Delta_0 \bar \ell_1 \varepsilon} \
    \min_{m\in[n]}
    \left(\sum_{i=1}^{m}\frac{1}{\tau_i}\right)^{-1}
    \left(\frac{1}{p} +m\right) \\
    &\geq \frac{1}{24}\ \frac{\bar L \Delta \min\{\frac{\sqrt{2\varepsilon} \varsigma }{\sigma},\,1\} }{8  \Delta_0 \bar \ell_1 \varepsilon} \
    \min_{m\in[n]}
    \left(\sum_{i=1}^{m}\frac{1}{\tau_i}\right)^{-1}
    \left( \frac{\sigma^2}{2\varepsilon \varsigma^2 } +m\right) .
  \end{align*}
  Choosing $c' = 1 / (40 \bar{\ell}_1 \Delta_0)$ implies $\varepsilon \leq \frac{\bar{L}\Delta}{8}$, hence the conditions of Lemma~\ref{lem:stat_lower_bound_time} hold, and combining both bounds yields the desired result.

  In the latter case,
  $
    \left(\frac{1}{2}\left\lfloor
    \frac{\bar L \Delta \sqrt{p}}{2 \Delta_0 \bar \ell_1 \varepsilon}
    \right\rfloor-1\right) \leq 1,
  $
  and having 
  $$
    \varepsilon < \frac{\bar{L}\Delta}{40 \bar{\ell}_1 \Delta_0}
  $$
  precludes the option $p = 1$.
  Hence, the right-hand side in \eqref{eq:time-lb-intermediate} is smaller than the lower bound in \Cref{lem:stat_lower_bound_time} up to a universal constant.
  This completes the proof.
\end{proof}
Here we prove the lemma that was used in the proof above.
\begin{lemma}\label{lem:stat_lower_bound_time}
  Assume $\varepsilon \le \nicefrac{\bar L \Delta}{8}$.
  Consider Protocol~\ref{protocol} with $n$ workers and batch-time functions
  $\{\tau_i(\cdot)\}_{i=1}^n$ satisfying Assumption~\ref{ass:batch-time}.
  Write $\tau_i\eqdef \tau_i(1)$ and assume
  $0<\tau_1\le \cdots \le \tau_n$.
  Then there exist functions $\{f_s\}_{s\in\{-1,+1\}}\subset \cF_{\Delta,\bar L}$ and, for each $s\in\{-1,+1\}$,
  a collection of oracles and distributions $((O_1,\dots,O_n),(\cD_1,\dots,\cD_n))\in
  \cO_{\tau_1(\cdot),\dots,\tau_n(\cdot)}^{\sigma^2,\bar L, B}(f_s)$
  such that for every algorithm $\mathsf A\in\cA$ run under the protocol,
  for any time $t\ge 0$,
  \begin{align}
    \max_{s\in\{-1,+1\}}
    \mathbb{E}_s\!\left[\ \inf_{(r,k)\in S_t}\ \|\nabla f_s(x^{(r,k)})\|^2\ \right]
    \;\ge\;
    \min\!\left\{\frac{\sigma^2}{64\,N(t)},\ \frac{\bar L\Delta}{8}\right\},
    \label{eq:stat_time_lb_sq_main}
  \end{align}
  where $\mathbb{E}_s$ denotes expectation under instance $s$, $S_t$ is the set of query-point indices whose corresponding gradients have been returned by time $t$, and
  \begin{align}
    N(t)\ \eqdef\ \sum_{i=1}^n \left\lfloor \frac{t}{\tau_i}\right\rfloor ,
    \label{eq:Nt_def}
  \end{align}
  with the convention $\sigma^2/(64\,N(t))\eqdef +\infty$ when $N(t)=0$.
  Consequently, if $t$ satisfies $N(t) \le \frac{\sigma^2}{64\varepsilon}$, then
  \begin{align}
    \max_{s\in\{-1,+1\}}
    \mathbb{E}_s\!\left[\ \inf_{(r,k)\in S_t}\ \|\nabla f_s(x^{(r,k)})\|^2\ \right]
    \;\ge\;\varepsilon .
    \label{eq:stat_time_lb_sq}
  \end{align}
  In particular, any algorithm satisfying
  $$
    \max_s \mathbb{E}_s\!\left[ \inf_{(r,k)\in S_t}\|\nabla f_s(x^{(r,k)})\|^2\right]\le \varepsilon
  $$
  must have
  \begin{align}
    t\ \ge\ c_0\min_{m\in[n]} \left(\sum_{i=1}^{m}\frac{1}{\tau_i}\right)^{-1}
    \left(\frac{\sigma^2}{\varepsilon}+m\right) ,
    \label{eq:stat_time_lb_final}
  \end{align}
  for some absolute constant $c_0$.
\end{lemma}
\begin{proof}
The argument is the same as \citet[Lemma~11]{arjevani2022lower}, with the sample budget
$T$ replaced by the maximal number of independent oracle draws available by time $t$, namely $N(t)$.

Fix $r\in\bigl(0,\sqrt{2\Delta/\bar L}\bigr)$ and define, for $s\in\{-1,+1\}$,
\begin{align*}
P_\xi^{\,s}\ \eqdef\ \mathcal N\!\left(rs,\frac{\sigma^2}{\bar L^2}\right),
\qquad f(x,\xi)\ \eqdef\ \frac{\bar L}{2}\Bigl(\|x\|^2-2\xi x_1+r^2\Bigr),
\qquad f_s(x)\ \eqdef\ \mathbb{E}_{\xi\sim P_\xi^{\,s}}[f(x,\xi)].
\end{align*}
Let $\theta_s\eqdef (rs,0,\dots,0)$. Then
\begin{align*}
f_s(x)
=\frac{\bar L}{2}\|x-\theta_s\|^2,
\end{align*}
so $f_s$ is $\bar L$--smooth and
$f_s(0)-\inf_x f_s(x)=\frac{\bar L}{2}r^2\le \Delta$, hence $f_s\in\cF_{\Delta,\bar L}$.

For each worker $i\in[n]$, set $\cD_i\eqdef P_\xi^{\,s}$ and let the stochastic gradient mapping be
$\nabla f(x;\xi) \eqdef \nabla_x f(x,\xi)=\bar L(x-\xi e_1)$, so that
\begin{align*}
  &\ExpSub{\xi}{\nabla f(x;\xi)} = \nabla f_s(x), \\
  &\ExpSub{\xi}{ \sqnorm{ \nabla f(x;\xi)-\nabla f_s(x) } } = \sigma^2, \\
  &\ExpSub{\xi}{ \sqnorm{ \nabla f(x;\xi)-\nabla f(y;\xi) } } = \bar L^2 \sqnorm{ x-y }.
\end{align*}
Thus, the corresponding oracles $O_i=O_{\tau_i(\cdot), B}^{f}$ (defined in \eqref{eq:time_k-batch-oracle-variable})
belong to $\cO_{\tau_1(\cdot),\dots,\tau_n(\cdot)}^{\sigma^2, \bar L, B}(f_s)$.

In \eqref{eq:time_k-batch-oracle-variable}, each \emph{completed} reply uses a single sample $\xi\sim\cD_i$ shared across the
entire returned batch.
Hence, a completed reply contributes at most one independent sample, regardless of the chosen batch size
$b\le B$. Since $\tau_i(\cdot)$ is nondecreasing (Assumption~\ref{ass:batch-time}), using $b>1$ cannot increase the number of independent samples received by time $t$ and can only increase completion times.
Therefore, for the purpose of lower bounding the obtainable information by time $t$, we may restrict attention to unit batches and work with $\tau_i=\tau_i(1)$, which yields the sample budget $N(t)$ in \eqref{eq:Nt_def}.

Let $S$ be uniform on $\{-1,+1\}$, indicating which instance is selected.
Conditioned on $S=s$, each completed reply reveals one independent draw from $P_\xi^{\,s}$ (indeed, from any returned gradient
$g=\bar L(x-\xi e_1)$ we recover $\xi=x_1-g_1/\bar L$ exactly).
By time $t$, worker $i$ can complete at most
$\lfloor t/\tau_i\rfloor$ such replies, hence the algorithm can receive at most $N(t)$ independent samples.

Define, for $s\in\{-1,+1\}$,
$$
  A_s \ \eqdef\ \inf_{(r,k)\in S_t}\ \norm{ \nabla f_s(x^{(r,k)}) }.
$$
Define
\begin{align}
  \hat S\ \eqdef\ \argmin_{s'\in\{-1,+1\}} A_{s'},
  \quad\text{with ties broken arbitrarily.}
  \label{eq:Shat_def}
\end{align}
If $\hat S\neq S$, then $A_{\hat S}\le A_S$, and hence
\begin{align*}
  2A_S \ge\ A_{1}+A_{-1}
  &\geq \inf_{x \in \mathbb{R}^d}\left(\|\nabla f_{1}(x)\|+\|\nabla f_{-1}(x)\|\right)\\
  &= \bar{L} \inf_{x \in \mathbb{R}^d}\left(\|x-\theta_1\|+\|x-\theta_{-1}\|\right)
  \geq \bar{L} \|\theta_1-\theta_{-1}\| = 2r\bar{L} ~.
\end{align*}
Hence, $A_S\ge r\bar L$ whenever $\hat S\neq S$, and therefore
\begin{align}
  \E{A_S}
  \ \ge\ r\bar L\cdot \mathbb{P}(\hat S\neq S),
  \label{eq:basic_markov}
\end{align}
where $\mathbb{P}$ is over the randomness of $S$, the oracle, and the algorithm.

Write $\mathbb{P}_s$ for the law of the information available by time $t$ under instance $s$.
Any estimator of $S$ based on this information has error at least
$$
  \mathbb{P}(\hat S\neq S)
  \ \ge\ \frac12\Bigl(1-\|\mathbb{P}_1-\mathbb{P}_{-1}\|_{\mathrm{TV}}\Bigr).
$$
By Pinsker's inequality,
$$
  \|\mathbb{P}_1-\mathbb{P}_{-1}\|_{\mathrm{TV}}
  \ \le\ \sqrt{\frac12 D_{\mathrm{KL}}(\mathbb{P}_1\|\mathbb{P}_{-1})} ~.
$$
Since the transcript by time $t$ is a measurable function of at most $N(t)$ i.i.d.\ samples
$\xi_1,\dots,\xi_{N(t)}$ with $\xi_j\sim P_\xi^{\,s}$ under instance $s$, the data-processing inequality yields
$$
  D_{\mathrm{KL}}(\mathbb{P}_1\|\mathbb{P}_{-1})
  \ \le\ D_{\mathrm{KL}}\!\left((P_\xi^{\,1})^{\otimes N(t)} \,\big\|\, (P_\xi^{\,-1})^{\otimes N(t)}\right)
  = N(t)\,D_{\mathrm{KL}}(P_\xi^{\,1}\|P_\xi^{\,-1}) ~.
$$
Moreover,
$$
  D_{\mathrm{KL}}(P_\xi^{\,1}\|P_\xi^{\,-1})
  =
  D_{\mathrm{KL}}\!\left(\mathcal N\!\left(r,\frac{\sigma^2}{\bar L^2}\right)
  \Big\| \ \mathcal N \left(-r,\frac{\sigma^2}{\bar L^2}\right)\right)
  = \frac{2r^2\bar L^2}{\sigma^2} ~.
$$
Thus,
$$
  \mathbb{P}(\hat S\neq S)\ \ge\ \frac12\left(1-\frac{r\bar L\sqrt{N(t)}}{\sigma}\right),
$$
and by \eqref{eq:basic_markov},
$$
  \mathbb{E}[A_S]
  \ \ge\ \frac{r\bar L}{2}\left(1-\frac{r\bar L\sqrt{N(t)}}{\sigma}\right).
$$
Now set
$$
  r\ \eqdef\ \min\left\{\frac{\sigma}{2\bar L\sqrt{N(t)}},\ \sqrt{\frac{2\Delta}{\bar L}}\right\},
$$
interpreting $\sigma/(2\bar L\sqrt{N(t)})\eqdef +\infty$ when $N(t)=0$.
Then $r\le \sqrt{2\Delta/\bar L}$ ensures $f_s(0)-\inf f_s\le \Delta$, and
$r\le \sigma/(2\bar L\sqrt{N(t)})$ makes the parenthesis at least $1/2$, so
$$
  \mathbb{E}[A_S]
  \ \ge\ \frac{r\bar L}{4}
  = \min\!\left\{\frac{\sigma}{8\sqrt{N(t)}},\ \sqrt{\frac{\bar L\Delta}{8}}\right\}.
$$
Since $\mathbb{E}[A_S]=\frac12\sum_{s\in\{\pm1\}}\mathbb{E}_s[A_s]$, we get
$$
  \max_{s\in\{\pm1\}}\mathbb{E}_s\!\left[\inf_{(r,k)\in S_t}\|\nabla f_s(x^{(r,k)})\|\right]
  \ \ge\ \min\!\left\{\frac{\sigma}{8\sqrt{N(t)}},\ \sqrt{\frac{\bar L\Delta}{8}}\right\}.
$$
Finally, for each $s$, Jensen's inequality yields
$$
  \mathbb{E}_s\!\left[\inf_{(r,k)\in S_t}\|\nabla f_s(x^{(r,k)})\|^2\right]
  =
  \ExpSub{s}{A_s^2}
  \ \ge\ \left( \ExpSub{s}{A_s} \right)^2,
$$
and squaring the previous bound gives \eqref{eq:stat_time_lb_sq_main}. The implication \eqref{eq:stat_time_lb_sq} follows
immediately from $\varepsilon \le \nicefrac{\bar L\Delta}{8}$ and $N(t)\le \sigma^2/(64\varepsilon)$.

It remains to obtain \eqref{eq:stat_time_lb_final}. Set $S_\varepsilon \eqdef \sigma^2/(64\varepsilon)$.
If $S_\varepsilon < \frac14$, then $\varepsilon > \sigma^2/16$ and any algorithm with
$$
  \max_s \mathbb{E}_s\!\left[\inf_{(r,k)\in S_t}\|\nabla f_s(x^{(r,k)})\|^2\right]\le \varepsilon
$$
must have $S_t\neq\emptyset$, hence $t\ge \tau_1$.
Moreover,
$$
  \frac{1}{384}\min_{m\in[n]} \left(\sum_{i=1}^{m}\frac{1}{\tau_i}\right)^{-1}
  \left(\frac{\sigma^2}{\varepsilon}+m\right)
  \ \le\ \frac{1}{384}\,\tau_1\left(\frac{\sigma^2}{\varepsilon}+1\right)
  \ <\ \tau_1,
$$
so \eqref{eq:stat_time_lb_final} holds.

Assume now that $S_\varepsilon\ge \frac14$.
Define $\tau_{n+1}\eqdef \infty$ and
$$
  j_\varepsilon^\star \ \eqdef\ \inf\left\{m\in[n] \,\middle|\, S_\varepsilon\left(\sum_{i=1}^{m}\frac{1}{\tau_i}\right)^{-1} < \tau_{m+1}\right\}.
$$
Define
$$
  t_1 \ \eqdef\ S_\varepsilon\left(\sum_{i=1}^{j_\varepsilon^\star}\frac{1}{\tau_i}\right)^{-1},
  \qquad
  t_2 \ \eqdef\ \min_{m\in[n]} \left(\sum_{i=1}^{m}\frac{1}{\tau_i}\right)^{-1}\left(S_\varepsilon+m\right).
$$
By \citep[Lemma~D.7]{tyurin2023optimal} (applied with $S=S_\varepsilon$), we have $t_1\le t_2\le 6t_1$.
Moreover, $t_1<\tau_{j_\varepsilon^\star+1}$, and hence $\lfloor t_1/\tau_i\rfloor=0$ for all $i\ge j_\varepsilon^\star+1$, so
$$
  N(t_1)
  =\sum_{i=1}^{j_\varepsilon^\star}\left\lfloor \frac{t_1}{\tau_i}\right\rfloor
  \le \sum_{i=1}^{j_\varepsilon^\star}\frac{t_1}{\tau_i}
  =t_1\sum_{i=1}^{j_\varepsilon^\star}\frac{1}{\tau_i}
  =S_\varepsilon.
$$
Since $N(\cdot)$ is nondecreasing, for any $t\le t_2/6$ we have $t\le t_1$ and thus
$N(t)\le N(t_1)\le S_\varepsilon$.
Therefore, any algorithm with $\max_s \mathbb{E}_s[\inf_{k\in S_t}\|\nabla f_s(x^{(k)})\|^2]\le \varepsilon$
must satisfy $t\ge t_2/6$, i.e.
$$
  t\ \ge\ \frac{1}{6}\min_{m\in[n]} \left(\sum_{i=1}^{m}\frac{1}{\tau_i}\right)^{-1}
  \left(\frac{\sigma^2}{64\varepsilon}+m\right) .
$$
Finally, since $\frac{\sigma^2}{64\varepsilon} + m \ge \frac{1}{64}\left(\frac{\sigma^2}{\varepsilon}+m\right)$ for all $m\in[n]$, we obtain \eqref{eq:stat_time_lb_final}.
\end{proof}

\newpage
\section{Additional experiments}
\label{add_experiments}
\begin{figure}[t] 

    \centering
    \begin{subfigure}[t]{0.32\textwidth}
        \centering
        \includegraphics[width=\textwidth]{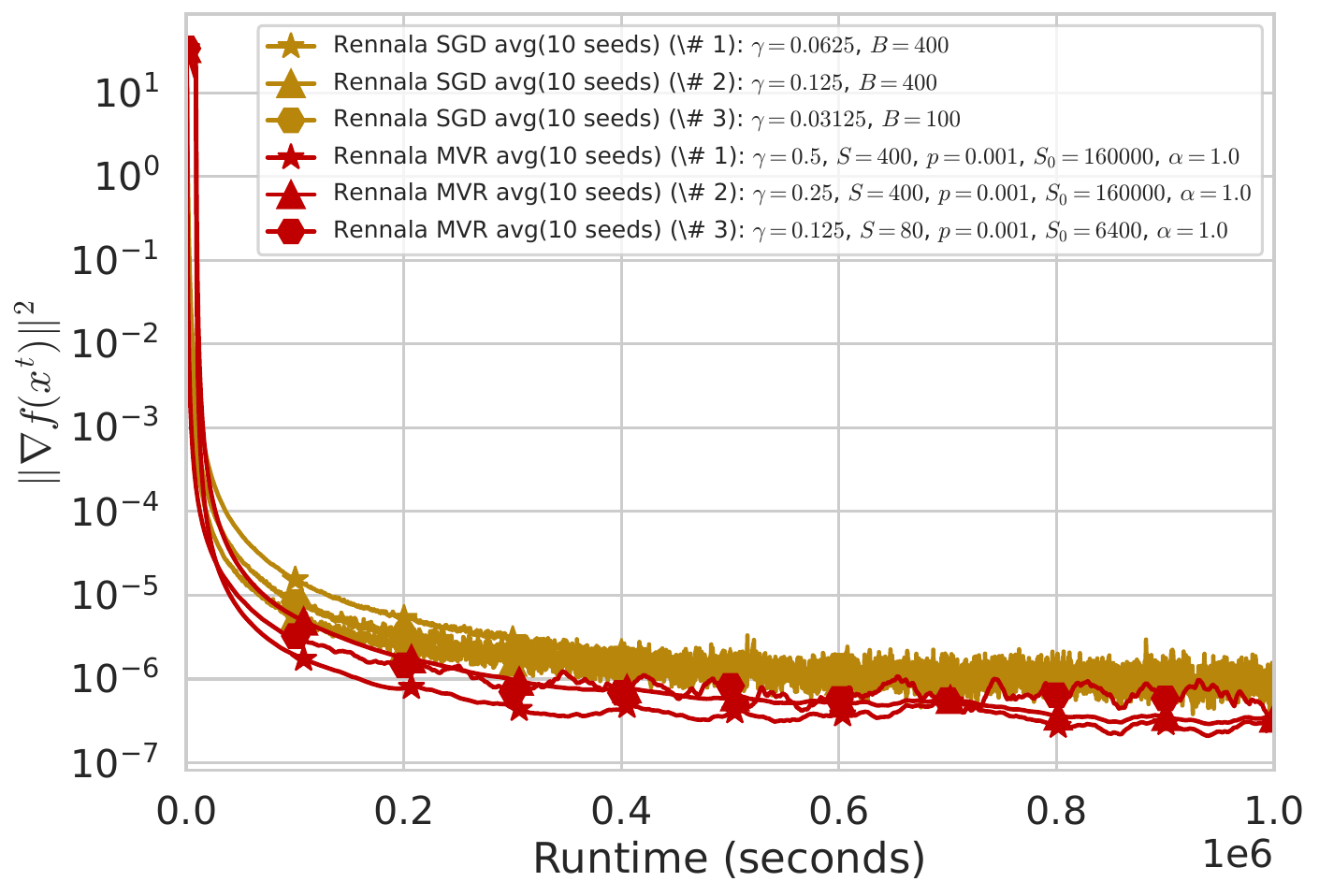}
        \caption{$\tau_i=\sqrt{i}$}
    \end{subfigure}
    \hfill
    \begin{subfigure}[t]{0.32\textwidth}
        \centering
        \includegraphics[width=\textwidth]{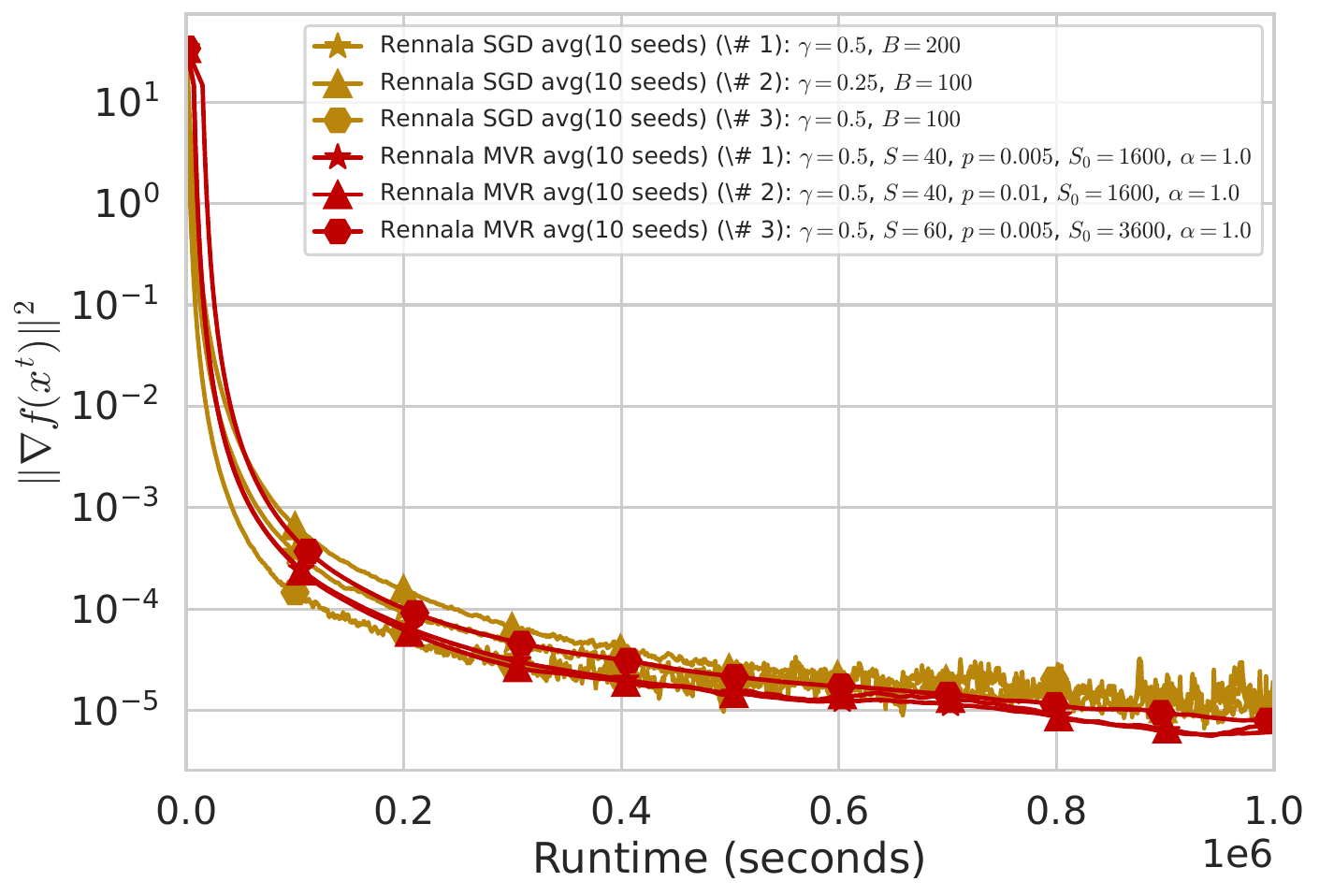}
        \caption{Mixture delays}
    \end{subfigure}
    \hfill
    \begin{subfigure}[t]{0.32\textwidth}
        \centering
        \includegraphics[width=\textwidth]{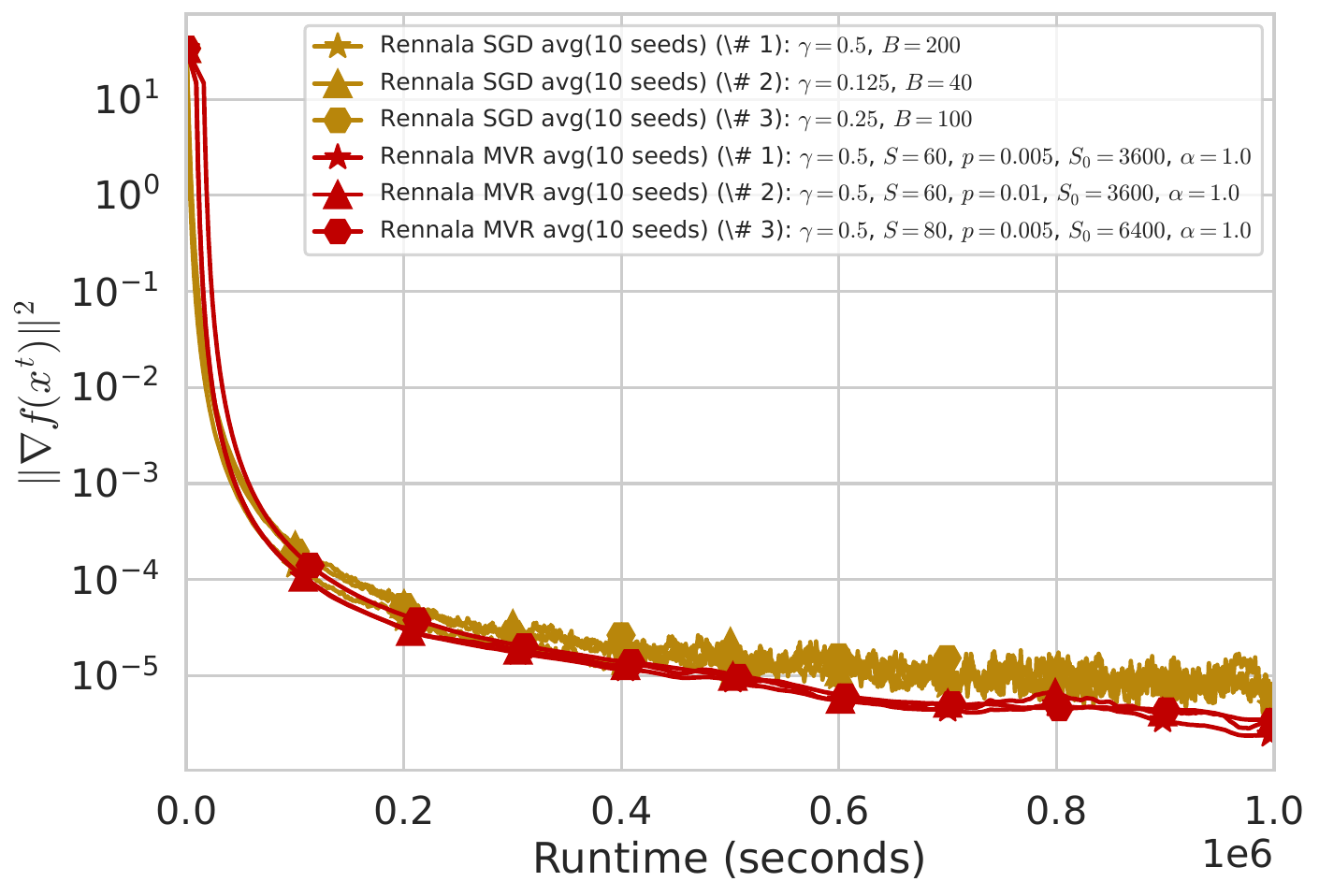}
        \caption{Uniform delays}
    \end{subfigure}
    \caption{Comparison of \rennalamvr and \rennala on the stochastic quadratic benchmark with $100$ workers under three delay models.}
    \label{fig:quadratic_100}
\end{figure}
We additionally performed experiments on the same stochastic quadratic benchmark with a larger number of workers, namely \(n=100\).
The corresponding plots in \Cref{fig:quadratic_100} show the same qualitative trend as in the \(n=10\) setting: \rennalamvr consistently outperforms \rennala across all considered delay models. 

We also examine the sensitivity of \rennalamvr to its additional hyperparameters. For the quadratic benchmark, the only extra parameter is \(p\). The heatmaps in Figure~\ref{fig:quadratic_sensitivity} show that, over the tested range, the method is relatively insensitive to the precise choice of \(p\). This is encouraging from a practical perspective, since it suggests that the additional flexibility of \rennalamvr does not translate into a substantial tuning burden.

We perform a similar analysis for the neural-network experiments, where the inexact variant introduces both \(p\) and the scaling parameter \(\alpha\). We first plot the full heatmaps for fixed values of \(\alpha\); see Figure~\ref{fig:nn_sensitivity_full}. These plots exhibit a pattern similar to that observed in the quadratic case, namely a relatively weak dependence on \(p\). We therefore additionally minimize over \(p\) and plot the resulting heatmaps as functions of the stepsize and \(\alpha\); see Figure~\ref{fig:nn_sensitivity_minimized}. The resulting plots indicate that, within the tested range, the method is also relatively insensitive to the choice of \(\alpha\), which further supports its practical viability.

\begin{figure*}[h]
    \centering
    \includegraphics[width=\textwidth]{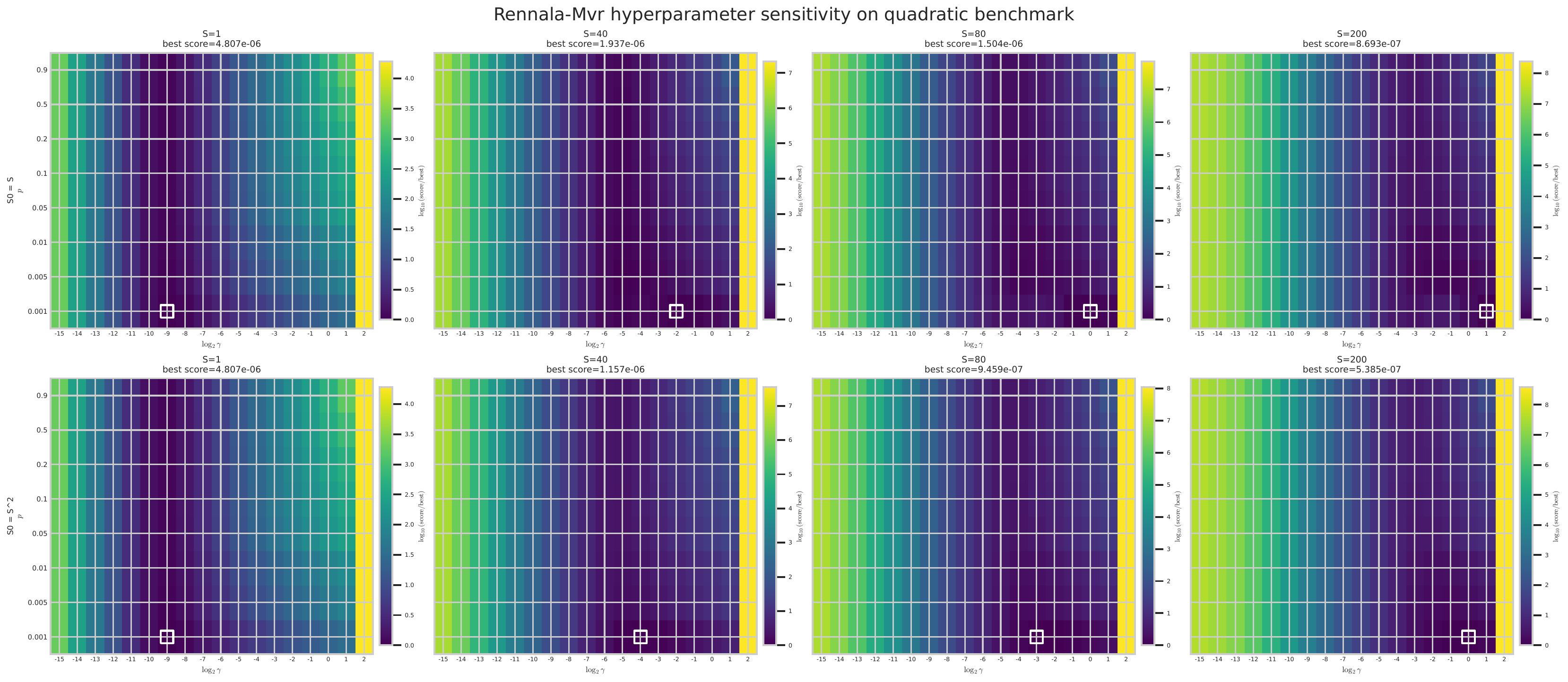}
    \caption{Sensitivity of exact \rennalamvr on the stochastic quadratic benchmark under square-root delays \(\tau_i=\sqrt{i}\). Each heatmap shows the performance criterion over the \((\gamma,p)\) grid for a fixed choice of \(B\) and \(B_0\).}
    \label{fig:quadratic_sensitivity}
\end{figure*}

\begin{figure*}[t]
    \centering
    \begin{subfigure}[t]{0.49\textwidth}
        \centering
        \includegraphics[page=1,width=\textwidth]{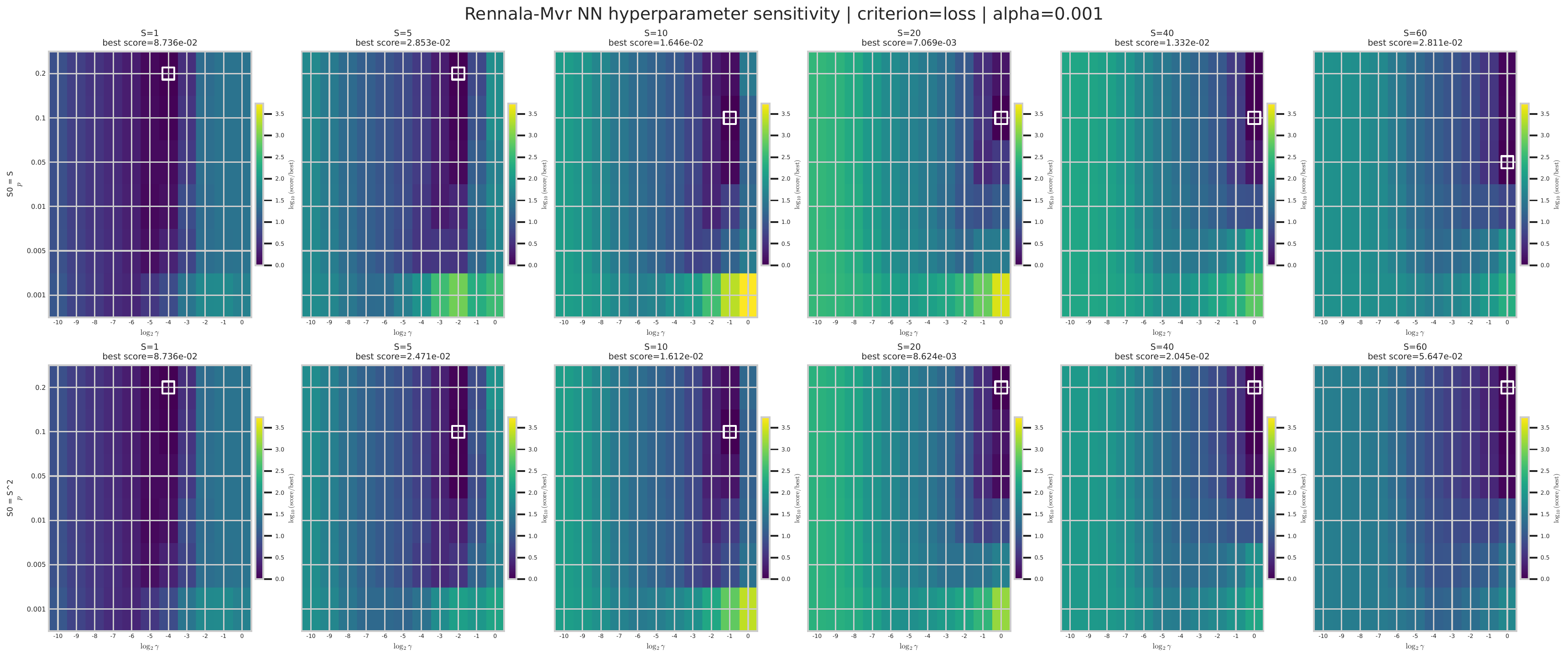}
        \caption{$\alpha=0.001$}
    \end{subfigure}\hfill
    \begin{subfigure}[t]{0.49\textwidth}
        \centering
        \includegraphics[page=2,width=\textwidth]{plots/sensitivity/nn_mvr_ablation_heatmaps.pdf}
        \caption{$\alpha=0.005$}
    \end{subfigure}

    \vspace{0.5em}

    \begin{subfigure}[t]{0.49\textwidth}
        \centering
        \includegraphics[page=3,width=\textwidth]{plots/sensitivity/nn_mvr_ablation_heatmaps.pdf}
        \caption{$\alpha=0.01$}
    \end{subfigure}\hfill
    \begin{subfigure}[t]{0.49\textwidth}
        \centering
        \includegraphics[page=4,width=\textwidth]{plots/sensitivity/nn_mvr_ablation_heatmaps.pdf}
        \caption{$\alpha=0.025$}
    \end{subfigure}

   \caption{Sensitivity of inexact \rennalamvr on asynchronous neural-network training under square-root delays \(\tau_i=\sqrt{i}\). Each panel corresponds to a fixed value of \(\alpha\) and shows the performance over the \((\gamma,p)\) grid for different choices of \(B\) and \(B_0\).}
    \label{fig:nn_sensitivity_full}
\end{figure*}

\begin{figure*}[h]
    \centering
    \includegraphics[width=\textwidth]{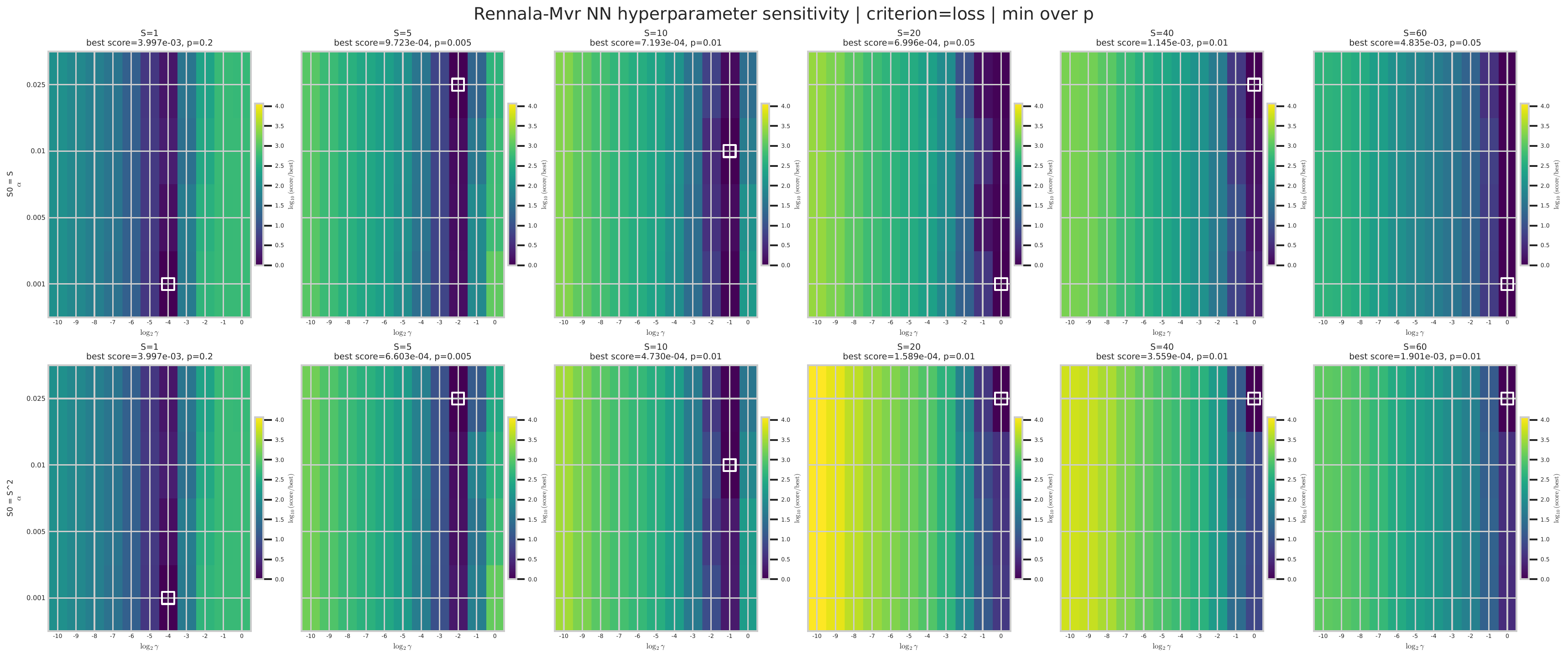}
  \caption{Sensitivity of inexact \rennalamvr on asynchronous neural-network training under square-root delays \(\tau_i=\sqrt{i}\) after minimizing over \(p\). Each heatmap shows the performance as a function of the stepsize and \(\alpha\) for a fixed choice of \(B\) and \(B_0\).}
      \label{fig:nn_sensitivity_minimized}
\end{figure*}

\end{document}

%% file: macros.tex
\usepackage{microtype}
\usepackage{graphicx}
\usepackage{subcaption}
\usepackage{booktabs} 
\usepackage{apptools}
\usepackage[flushleft]{threeparttable}
\usepackage{array,booktabs,makecell}
\usepackage{multirow}
\usepackage{nicefrac}
\usepackage{amsthm}
\usepackage{mathtools}
\usepackage{amsmath,amsfonts,bm,amssymb}
\usepackage{dsfont}
\usepackage{wrapfig}
\usepackage{caption}
\usepackage{siunitx}
\usepackage{thm-restate}
\usepackage{nccmath}
\usepackage{empheq}
\usepackage{bbm}
\usepackage{tabularx}

\usepackage{algorithmic}
\usepackage{algorithm}
\usepackage{filecontents}

\makeatletter
\newenvironment{protocol}[1][htb]{%
    \renewcommand{\ALG@name}{Protocol}
   \begin{algorithm}[#1]%
  }{\end{algorithm}}
\makeatother

\usepackage[capitalize,noabbrev]{cleveref}

\usepackage{color}
\usepackage{colortbl}
\definecolor{bgcolor}{rgb}{0.76,0.88,0.50}
\definecolor{bgcolor0}{rgb}{0.93,0.99,1}
\definecolor{bgcolor1}{rgb}{0.8,1,1}
\definecolor{bgcolor2}{rgb}{0.8,1,0.8}
\definecolor{bgcolor3}{rgb}{0.50,0.90,0.50}
\usepackage{tcolorbox}
\usepackage{pifont}

\definecolor{mydarkblue}{rgb}{0,0.16,0.54}
\definecolor{mydarkgreen}{RGB}{39,130,67}
\definecolor{mydarkorange}{RGB}{236,147,14}
\definecolor{mydarkgreen}{RGB}{0,100,0}
\definecolor{ruby}{RGB}{155,17,30}
\definecolor{chili}{RGB}{191,0,0}
\definecolor{sangria}{RGB}{146,0,10}
\definecolor{burgundy}{RGB}{128,0,32}
\definecolor{darkred}{RGB}{132,0,0} 
\definecolor{cherry}{RGB}{192,0,0} 
\definecolor{myaccent}{rgb}{0,0.45,0.45}
\definecolor{anotherdarkred}{rgb}{0.45,0,0.05}

\newcommand{\orange}{\color{mydarkorange}}

\newcommand{\red}{\color{cherry}}

\usepackage{xspace}
\usepackage{relsize}

\newcommand{\norm}[1]{\left\| #1 \right\|}
\newcommand{\sqnorm}[1]{\left\| #1 \right\|^2}

\newcommand{\R}{\mathbb{R}} 
\newcommand{\E}[1]{\mathbb{E}\left[#1\right]}

\newcommand{\progg}{\mathrm{prog}}

\newcommand{\ExpSub}[2]{{\mathbb{E}}_{#1}\left[#2\right]}
\newcommand{\ExpCond}[2]{{\mathbb{E}}\left[\left.#1\ \right\vert\ #2\right]}

\newcommand{\cA}{\mathcal{A}}

\newcommand{\cD}{\mathcal{D}}

\newcommand{\cF}{\mathcal{F}}

\newcommand{\cO}{\mathcal{O}}

\newcommand{\eqdef}{\coloneqq} 

\newcommand{\minimize}{\mathop{\mathrm{minimize}}}

\makeatletter
\newcommand{\vast}{\bBigg@{4}}

\DeclareMathOperator*{\argmin}{arg\,min}

\definecolor{myblue}{rgb}{0.3,0.25,0.2}
\definecolor{myorange}{rgb}{0.3,0.25,0.2}

\definecolor{alggray}{gray}{0.35} 
\newcommand{\algname}[1]{{\color{alggray}\small\sf#1}\xspace}

\newcommand{\ringmaster}{\algname{Ringmaster~ASGD}}

\newcommand{\rennalatitle}{Rennala~SGD}
\newcommand{\rennala}{\algname{Rennala~SGD}}
\newcommand{\sgd}{\algname{SGD}}

\newcommand{\mvr}{\algname{MVR}}
\newcommand{\storm}{\algname{STORM}}
\newcommand{\rennalamvrtitle}{{\red Rennala~MVR}}
\newcommand{\rennalamvr}{\algname{\red Rennala~MVR}}

\usepackage{thmtools}
\usepackage{thm-restate}
\usepackage{mdframed}

\theoremstyle{plain}
\newtheorem{theorem}{Theorem}[section]

\newtheorem{lemma}[theorem]{Lemma}

\theoremstyle{definition}
\newtheorem{definition}[theorem]{Definition}
\newtheorem{assumption}[theorem]{Assumption}
\theoremstyle{remark}

\theoremstyle{plain} 

\newmdenv[
  font=\normalfont\bfseries,
  linecolor=black,
  linewidth=0.8pt,
  topline=false,
  bottomline=false,
  leftline=false,
  rightline=false,
  backgroundcolor=gray!13,
  skipabove=2pt,
  skipbelow=2pt,
  innertopmargin=-5pt,
  innerbottommargin=4pt,
  innerrightmargin=4pt,
  innerleftmargin=4pt,
]{myboxed}

\newmdtheoremenv[
  font=\normalfont\bfseries,
  linecolor=black,
  linewidth=0.8pt,
  topline=false,
  bottomline=false,
  leftline=false,
  rightline=false,
  backgroundcolor=gray!13,
  skipabove=2pt,
  skipbelow=2pt,
  innertopmargin=-5pt,
  innerbottommargin=4pt,
  innerrightmargin=4pt,
  innerleftmargin=4pt,
]{boxedtheorem}[theorem]{Theorem}

\newmdtheoremenv[
  font=\normalfont\bfseries,
  linecolor=black,
  linewidth=0.8pt,
  topline=false,
  bottomline=false,
  leftline=false,
  rightline=false,
  backgroundcolor=gray!13,
  skipabove=2pt,
  skipbelow=2pt,
  innertopmargin=-5pt,
  innerbottommargin=4pt,
  innerrightmargin=4pt,
  innerleftmargin=4pt,
]{boxedlemma}[theorem]{Lemma}

\newmdtheoremenv[
  font=\normalfont\bfseries,
  linecolor=black,
  linewidth=0.8pt,
  topline=false,
  bottomline=false,
  leftline=false,
  rightline=false,
  backgroundcolor=gray!13,
  skipabove=2pt,
  skipbelow=2pt,
  innertopmargin=-5pt,
  innerbottommargin=4pt,
  innerrightmargin=4pt,
  innerleftmargin=4pt,
]{boxeddefinition}[theorem]{Definition}

\newmdtheoremenv[
  font=\normalfont\bfseries,
  linecolor=black,
  linewidth=0.8pt,
  topline=false,
  bottomline=false,
  leftline=false,
  rightline=false,
  backgroundcolor=gray!13,
  skipabove=2pt,
  skipbelow=2pt,
  innertopmargin=-5pt,
  innerbottommargin=4pt,
  innerrightmargin=4pt,
  innerleftmargin=4pt,
]{boxedassumption}[theorem]{Assumption}

\newtheorem{innercustomthm}{Theorem}
\newenvironment{restate-theorem}[1]
  {\renewcommand\theinnercustomthm{#1}\innercustomthm}
  {\endinnercustomthm}

\newtheorem{innercustomlemma}{Lemma}
\newenvironment{restate-lemma}[1]
  {\renewcommand\theinnercustomlemma{#1}\innercustomlemma}
  {\endinnercustomlemma}

\newtheorem{innercustomproposition}{Proposition}
\newenvironment{restate-proposition}[1]
  {\renewcommand\theinnercustomproposition{#1}\innercustomproposition}
  {\endinnercustomproposition}

\newenvironment{restate-boxedtheorem}[1]
  {\renewcommand\theinnercustomthm{#1}\begin{myboxed}\begin{innercustomthm}}
  {\end{innercustomthm}\end{myboxed}}

\newenvironment{restate-boxedlemma}[1]
  {\renewcommand\theinnercustomlemma{#1}\begin{myboxed}\begin{innercustomlemma}}
  {\end{innercustomlemma}\end{myboxed}}

\newcommand*{\sketchproofname}{Sketch of Proof}